\documentclass{siamart250211}
\usepackage[utf8]{inputenc}
\usepackage{amsmath,amssymb}
\usepackage{amsfonts}
\usepackage{graphicx}
\usepackage{mathrsfs}
\usepackage{hyperref}
\usepackage{subfigure}
\hypersetup{
	colorlinks=true,
	linkcolor=blue,
	anchorcolor=blue,
	citecolor=blue
}

\newtheorem{assumption}{Assumption}
\newtheorem{remark}{Remark}
\newtheorem{example}{Example}

\newcommand{\veps}{\varepsilon}
\newcommand{\psie}{\psi^{\veps}}
\newcommand{\hpsie}{\hat{\psi}^{\veps}}
\newcommand{\psieo}{\varphi^{\veps}_0}
\newcommand{\hpsieo}{\hat{\varphi}^{\veps}_0}

\newcommand{\ze}{\zeta}
\newcommand{\dbx}{\,\mathrm{d}x}
\newcommand{\dby}{\,\mathrm{d}y}

\newcommand{\dbp}{\,\mathrm{d}p}
\newcommand{\dbq}{\,\mathrm{d}q}

\newcommand{\dbP}{\,\mathrm{d}P}
\newcommand{\dbQ}{\,\mathrm{d}Q}
\newcommand{\dxi}{\,\mathrm{d}\xi}
\newcommand{\dze}{\,\mathrm{d}\ze}
\newcommand{\dtau}{\,\mathrm{d}\tau}

\newcommand{\amp}{{a}}

\newcommand{\p}{\partial}
\newcommand{\f}[2]{\frac{#1}{#2}}
\newcommand{\pa}{\partial_a}

\newcommand{\pab}{\partial^2_{ab}}

\newcommand{\ddt}[1]{\frac{\rd{#1}}{\rd t}}
\newcommand{\abs}[1]{\left|{#1}\right|}
\newcommand{\norm}[1]{\left\|{#1}\right\|}

\newcommand{\e}{\mathrm{e}}
\newcommand{\ri}{\mathrm{i}}
\newcommand{\rd}{\mathrm{d}}
\newcommand{\bbR}{\mathbb{R}}
\newcommand{\bbC}{\mathbb{C}}
\newcommand{\N}{\mathbb{N}}

\newcommand{\FIO}{\mathcal{I}^{\veps}_{\Phi}}
\newcommand{\tFIO}{\mathcal{I}^{\veps}_{\tilde\Phi}}
\newcommand{\FGA}{\mathrm{FGA}}
\newcommand{\FFT}{\mathcal{F}}
\newcommand{\IFT}{\mathcal{F}^{-1}}
\newcommand{\FGAdt}{{\delta,\FGA}}
\newcommand{\kin}{T}
\newcommand{\kindt}{\kin^\delta}
\newcommand{\dkin}[1]{T^{#1}}
\newcommand{\dkindt}[1]{T^{\delta,#1}}
\newcommand{\psiedt}{\psie_\delta}
\newcommand{\FBI}{\mathscr{F}^{\veps}}
\newcommand{\FSE}{FSE}
\newcommand{\FSEs}{fractional Schr\"odinger equations}
\newcommand{\Tf}{t_\mathrm{f}}

\newcommand{\etal}{{et al.}\;}
\newcommand{\ie}{{i.e.}\;}
\newcommand{\eg}{{e.g.}\;}

\newcommand{\emr}{(r^{d+4}+1)\e^{-r^2}} 

\newcommand{\red}[1]{\textcolor{red}{#1}}

\title{Frozen Gaussian approximation for the fractional Schr\"odinger equation}
\author{
	Lihui Chai\footnote{School of Mathematics, Sun Yat-sen University, Guangzhou, 510275, China (chailihui@mail.sysu.edu.cn).} 
	\and Hengzhun Chen\footnote{School of Mathematical Sciences, Fudan University, Shanghai, 200433, China (hengzhunchen21@m.fudan.edu.cn).}
	\and Xu Yang\footnote{Department of Mathematics, University of California, Santa Barbara, CA 93106, USA (xuyang@math.ucsb.edu).}
}

\begin{document}

\numberwithin{equation}{section}
\numberwithin{figure}{section}
\numberwithin{theorem}{section}

\maketitle
\begin{abstract}
    We develop a refined Frozen Gaussian approximation (FGA) for the fractional Schr\"odinger equation in the semi-classical regime, where the solution exhibits rapid oscillations as the scaled Planck constant $\varepsilon$ becomes small. Our approach utilizes an integral representation based on asymptotic analysis, offering a highly efficient computational framework for high-frequency wave function evolution. Crucially, we introduce the momentum space representation of the FGA and a regularization parameter $\delta$ to address singularities in the higher-order derivatives of the Hamiltonian, which are typically assumed to be smooth in conventional analysis. We rigorously prove convergence of the method to the true solution and provide numerical experiments that demonstrate its precision and robust convergence behavior.
\end{abstract}

\section{Introduction}
We consider the Schr\"odinger equation
\begin{equation} \label{eq:Schrodinger}
    \ri\veps \partial_t \psie(t,x)= \hat{H}(x,-\ri\veps\partial_x)\psie(t,x),
    \quad x\in\bbR^d,\;t>0,
\end{equation}
where $0<\veps\ll1$ indicates the semi-classical regime, $\psie(t,x)$ is the complex-valued wavefunction, and the Hamiltonian operator $\hat{H}$ is defined by a pseudo-differential operator \cite{gerard1997homogenization,hormander1987analysis} associated with a symbol $H:\bbR^{2d}\rightarrow\bbR$,
\begin{equation} \label{eq:pseudo_Ham}
    \hat{H}(x,-\ri\veps\partial_x) \psie(x):=\frac{1}{(2 \pi)^{d}} \int_{\mathbb{R}^{2d}} \,\e^{\ri(x-y)\cdot\xi} H\left(\f{x+y}{2},\veps\xi\right) \psie(y) \dxi\dby.
\end{equation}
We assume that the symbol of the Hamiltonian takes the form $H(x,\xi)=\kin(\xi)+V(x)$,
where $\kin(\xi)$ and $V(x)$ represent the kinetic and potential energy symbols, respectively.

For example, when $\displaystyle\kin(p)=\f{p^2}{2}$, 
one has the standard semi-classical Schr\"odinger equation. When $\displaystyle\kin(p)=\f{|p|^\alpha}{\alpha}$ for $1 \leq \alpha \leq 2$, 
one has the fractional Schr\"odinger equation (\FSE),
\begin{equation} \label{eq:FSE}
    \ri\veps\p_t\psie = -\frac{\veps^{\alpha}}{\alpha}\Delta^{\alpha/2}\psie + V(x)\psie.
\end{equation}
In general, the kinetic operator is given by
\begin{multline} \label{eq:kin-operator}
    \kin(-\ri\veps\p_x) \psi(x)
    :=\IFT\left[\kin(\xi)\hat{\psi}(\xi)\right](x)
    \\= \frac{1}{(2 \pi \veps)^{d}} \int_{\mathbb{R}^{d}} \kin(\xi) \e^{\ri x \xi/\veps} \hat{\psi}(\xi)\dxi
    = \frac{1}{(2 \pi \veps)^{d/2}} \int_{\mathbb{R}^{2d}} \kin(\xi) \e^{\ri(x-y) \xi/\veps} \psi(y)\dby\dxi.
\end{multline}
In the above equations, we have used the Fourier transformation and its inverse transformation denoted as follows:
\begin{align}
     & \FFT[\psi](\xi)  =\hat\psi(\xi) = {(2\pi\veps)^{-\f{d}2}}\!\int_{\bbR^d}\psi(x)\,\e^{-\ri x \xi / \veps} \dbx, \\
     & \IFT[\hat\psi](x) = {(2\pi\veps)^{-\f{d}2}}\!\int_{\bbR^d}\hat\psi(\xi)\,\e^{\ri x \xi / \veps} \dxi.
\end{align}

The semi-classical Schr\"odinger equation \eqref{eq:Schrodinger} has been extensively studied in both theoretical and numerical aspects. Notable theoretical contributions include works on WKB analysis and Wigner measures, such as those by Tartar \cite{tartar1990h}, Lions and Paul \cite{lions1993mesures}, G\'erard \etal \cite{gerard1997homogenization}, and Carles \cite{carles2008semi}. On the numerical side, methods such as time-splitting spectral methods and splitting schemes have been developed by Bao \etal \cite{bao2002time}, Engquist and Runborg \cite{EnRu:03}, Gosse and Jin \cite{gosse2003two}, Lubich \cite{lubich2008splitting}, and Jin \etal \cite{jin2011mathematical}.  These methods typically assume that the Hamiltonian symbol is a polynomial in $\xi$ or possesses a certain degree of smoothness, often up to $C^\infty$. However, for \FSE, when $1\le\alpha<2$, the symbol is only in $C^\alpha$. Although similar methods can be applied formally, their convergence remains unclear due to the lower regularity of the symbol. This poses a significant challenge in both the theoretical analysis and the development of robust numerical methods for the \FSE.

The \FSE~was first introduced by Laskin in \cite{laskin2000fractionalE,laskin2002fractional,laskin2000fractionalA}, where the path integral was generalized from Brownian-type quantum mechanics trajectory to L\'evy-type quantum mechanics trajectory, to represent the Bohr atom, and fractional oscillator, and to study the quantum chromodynamics (QCD) problem of quarkonium.
The \FSE~with $\alpha=1$ can also serve as a toy model for the mathematical description of the dynamics of semi-relativistic boson stars in the mean-field limit \cite{frohlich2007boson,lenzmann2007well}.
Numerically, the \FSE~has been studied using various methods. Finite difference methods have been employed \cite{HuYaOb2014}, and time-splitting spectral methods have been developed \cite{duo2016mass,wang2022lie}.
Properties of nonlinear \FSEs, such as global existence, the possibility of finite time blow-up, the existence and stability of the ground states, and the decoherence of solutions, have been explored using Fourier spectral methods \cite{klein_numerical_2014,antoine_ground_2016,kirkpatrick2016fractional}.
Additionally, numerical methods with perfectly matched layers (PMLs) in both real-space and Fourier-space have been proposed for linear \FSEs~\cite{antoine_derivation_2021}.


In this paper, we derive a frozen Gaussian approximation (FGA) for the \FSE~and show the convergence in the semi-classical regime. The FGA was originally used in quantum chemistry \cite{He:81,HeKl:84} and was systematically justified for the Schr\"odinger equation in \cite{Ka:94,Ka:06,SwRo:09,LaSa:17}. Subsequently, the FGA theory was extended to linear strictly hyperbolic systems by pioneering works \cite{LuYa:11,LuYa:CPAM} and more recently to non-strictly hyperbolic systems, including elastic wave equations \cite{hate2018fga}, relativistic Dirac equations in the semi-classical regime \cite{ChLo:19,Chai2021fgaconv}, and non-adiabatic dynamics in surface hopping problems \cite{LuZh:2016FGAsh,LuZh:2018FGAsh,Huang2022}.

The FGA propagates the wavefield through the classical ray center and complex-valued amplitude which are determined by the associated Hamiltonian flow. Given that the coefficients of the equation are sufficiently smooth (at least in $C^2$), one can derive the FGA equations along the ray path using asymptotic expansion with respect to the semi-classical parameter $\veps$ and prove the convergence in the first-order of $\veps$ (higher-order convergence can be achieved with further expansion). However, when some coefficient is not in $C^2$, such as in the \FSE~with $1<\alpha<2$, the expansion may break down, leading to the propagation of the ray path governed by the Hamiltonian system with singularities. In this scenario, the convergence of the FGA is not guaranteed by existing theory. To address this issue, we introduce a regularization parameter $\delta$ when deriving the FGA formulation. By carefully analyzing the residual terms in the expansion and tuning the parameter $\delta$ in conjunction with $\veps$, we can bound the high-order derivatives of the Hamiltonian flow and obtain a convergence result. Another challenge in dealing with the fractional Hamiltonian system is the difficulty in estimating certain terms corresponding to the fractional Laplacian when performing asymptotic expansion in position space. Therefore, we propose an asymptotic expansion for the fractional Laplacian in the momentum space (Lemma \ref{lemma:xi-P}), which yields a compact form of the remainder terms, facilitating convergence analysis.

The rest of this paper is organized as follows: Section 2 introduces the FGA for the \FSE, presenting preliminary results and a formal derivation while leaving the low-regularity issue of the Hamiltonian symbol unaddressed. In Section 3, we tackle the low-regularity issue using a regularization technique and rigorously analyze the convergence of the FGA. The convergence result depends on the spatial dimension $d$ and the fractional order $\alpha$. Numerical experiments demonstrating the performance of the FGA are presented in Section 4, and we conclude with remarks in Section 5.

\smallskip
\noindent \textbf{Notations.}
1. The absolute value, Euclidean distance, vector norm, induced matrix norm, and the sum of components of a multi-index will all be denoted by $|\cdot|$.

\smallskip
\noindent 2. For a vector $v\in\bbR^d$, $v^2_{jk}$ denotes the $(j,k)$-th element $v_jv_k$ of the rank 2 tensor $v\otimes v$, and $v^3_{jkl}$  denotes the $(j,k,l)$-th element $v_jv_kv_l$ of the rank 3 tensor $v\otimes v\otimes v$.

\smallskip
\noindent 3. We use $\mathcal{S}$ and $C^{\infty}$ to denote the Schwartz class functions and smooth functions, respectively.

\section{The frozen Gaussian approximation}\label{sec:FGA}
In this section, we introduce the necessary concepts, assumptions, and preliminary results that will be utilized in deriving the frozen Gaussian approximation (FGA). While some of these have been systematically introduced in existing literature, such as \cite{LuYa:11,LuYa:CPAM}, we provide a concise overview for clarity. Subsequently, we present the initial formulation of the FGA for the \FSE, and in Section 3, we will demonstrate its further modified version and convergence.

\subsection{Hamiltonian flow and action}\label{sec:Hflow}

The Hamiltonian flow plays a crucial role in the FGA. Given the Hamiltonian of \FSE~\eqref{eq:FSE}
\begin{equation} \label{eq:hamiltonian}
    H(Q, P) = \frac{1}{\alpha} |P|^{\alpha} + V(Q) := \kin(P) + V(Q).
\end{equation}
Let us introduce $\nabla H$, $\nabla^2 H$, and $J$ as follows:
$$
    \nabla H = \begin{pmatrix}\p_QV\\ \p_PT\end{pmatrix} , \quad
    \nabla^2 H = \begin{pmatrix}\p^2_QV & \\ &\p^2_PT\end{pmatrix}, \quad
    J= \begin{pmatrix} 0 & \mathrm{Id}_d \\ -\mathrm{Id}_d & 0 \\ \end{pmatrix},
$$
where $\mathrm{Id}_d$ is the $d$-by-$d$ identity matrix.
Then the Hamiltonian flow is defined by the map
$$
    \kappa(t): \quad
    \begin{aligned}
        \mathbb{R}^{2 d} & \rightarrow \mathbb{R}^{2 d}                \\
        (q, p)           & \mapsto \left(Q(t, q, p), P(t, q, p)\right)
    \end{aligned} \; ,
$$
where $(Q(t,q,p), P(t,q,p))$ satisfy the system of ODEs:
\begin{align}\label{eq:hamiltonian-flow}
    \ddt{}
    \begin{pmatrix}Q\\P\end{pmatrix} =  J\nabla H, \quad
    \left.\begin{pmatrix}Q\\P\end{pmatrix}\right|_{t=0}= \begin{pmatrix}q\\p\end{pmatrix}.
\end{align}
It can be shown (see \eg \cite{LuYa:CPAM}) that the map $\kappa(t)$ is a canonical transformation defined as follows:

\begin{definition}[Canonical Transformation] \label{def:canonical}
    Let $\kappa: \mathbb{R}^{2d}\rightarrow \mathbb{R}^{2d}$ be a differentiable map
    $
        \kappa(q, p)=(Q(q, p), P(q, p)).
    $
    We denote the Jacobian matrix as
    \begin{equation}\label{eq:F}
        F(q, p)= \begin{pmatrix}
            (\partial_q Q)^{T}(q, p) & (\partial_p Q)^{T}(q, p) \\
            (\partial_q P)^{T}(q, p) & (\partial_p P)^{T}(q, p) \\
        \end{pmatrix}.
    \end{equation}
    We say that $\kappa$ is a canonical transformation if for any $(q, p)\in \mathbb{R}^{2d}$, $F$ is symplectic , \ie
    \begin{equation}\label{eq:Fsympletic}
        F^T J F = J.
    \end{equation}
\end{definition}
For the convenience of the analysis of the FGA, we introduce the operator $\partial_z: = \partial_q - \ri \partial_p$ and the matrix $Z(t,q,p):= \partial_z \left(Q(t,q,p) + \ri P(t,q,p)\right)$ associated with canonical transformation $\kappa(q,p)$.
Moreover, we assert the invertibility of $Z$ and provide the definition of the action associated with Hamiltonian flow $(Q, P)$. For a detailed proof, we refer the reader to \cite{LuYa:CPAM}.
\begin{lemma} \label{lemma:invert_Z}
    $Z(t,q,p)$ is invertible for $(q,p)\in \bbR^{2d}$ with $|\det(Z(t,q,p)| \geq 2^{d/2}$.
\end{lemma}
\begin{definition}[Action] \label{def:action}
    Suppose $\kappa$ is a canonical transformation, then a function $S:\mathbb{R}^{2d}\rightarrow \mathbb{R}$ is called an action associated with $\kappa$ if it satisfies
    \begin{align}
        \partial_p S(q, p)  = \partial_p Q(q, p)\cdot P(q, p),
        \quad\text{ and }\quad
        \partial_q S(q, p)  = -p + \partial_q Q(q, p) \cdot P(q, p).
    \end{align}
\end{definition}
\noindent
It is straightforward to verify that the solution $S$ of the equation
\begin{equation}\label{eq:action}
    \ddt{S} = P \cdot \partial_P H(Q, P) - H(Q, P), \quad S(0, q, p)=0.
\end{equation}
is indeed an action associated with $\kappa(t)$.

\subsection{Fourier integral operator}\label{sec:FIO}

The FGA can be formulated using the Fourier integral operator, a commonly used definition in the relevant literature:
\begin{definition}[Fourier Integral Operator]
    For Schwartz-class functions $u\in \mathcal{S}(\mathbb{R}^{2d}; \mathbb{C})$ and $\varphi \in\mathcal{S}(\mathbb{R}^d; \mathbb{C})$, we define the Fourier Integral Operator with symbol $u$ as
    \begin{equation}
        [\FIO(u)\varphi](x)=\frac{1}{(2\pi\veps)^{3d/2}} \int_{\mathbb{R}^{3d}} \e^{\ri\Phi(t,x,y,q,p)/\veps} \,u(q,p)\,\varphi(y) \dby \dbq \dbp,
    \end{equation}
    where the complex-valued phase function $\Phi(t,x,y,q,p)$ is given by
    \begin{multline} \label{eq:Phi}
        \Phi(t,x,y,q,p)=S(t,q,p)+P(t,q,p)\cdot(x-Q(t,q,p))
        \\ + \frac{\ri}{2}|x-Q(t,q,p)|^2-p\cdot(y-q)+\frac{\ri}{2}|y-q|^2.
    \end{multline}
\end{definition}
In preparation for the subsequent derivation and convergence analysis, we say two functions are equivalent if they give the same Fourier integral. Precisely, we introduce the following definition:
\begin{definition} \label{def:sim}
    Given $f(q,p), g(q,p)$ in Schwartz class, we say that $f$ and $g$ are equivalent (under a given phase function $\Phi$), denoted as $f\stackrel{\Phi}{\sim} g$, if for any function $\varphi=\varphi(y)\in\mathcal{S}(\bbR^d;\bbC)$ it holds that $\FIO(f)\varphi = \FIO(g)\varphi$.
\end{definition}
\noindent
Furthermore, we state Lemma 5.2 in \cite{LuYa:CPAM}, which plays a crucial role in deriving the conventional FGA, utilizing the Einstein summation convention for repeated indices:
\begin{lemma} \label{lemma:x-Q}
    For any $v(y, q, p)\in\bbR^{d}$
    in Schwartz class, we have
    \begin{align}
        \label{eq:v(x-Q)}
        (x-Q)_j\,v_j \stackrel{\Phi}{\sim} &
        {}-\veps \partial_{z_{k}}\left(v_{j} Z_{j k}^{-1}\right).
    \end{align}
    Here $Z^{-1}$ denotes the inverse matrix of $Z$. More precisely, $Z_{jk}=\p_{z_j}(Q_k+\ri P_k)$, and the $(j,k)$-th entry of $Z^{-1}$ is denoted as $Z^{-1}_{jk}$.
\end{lemma}

This lemma involves utilizing integration by part with relation $x - Q=\ri Z^{-1} \partial_z \Phi$ to obtain the equivalent asymptotic orders with respect to $\veps$ for each term. Given an FGA formulation such as \eqref{eq:fga_standard}, one can derive the equations of FGA quantities to formulate the FGA solution that is accurate up to $O(\veps^k)$ by aligning the asymptotic expansions. In conventional FGA analysis, the typical approach is to employ the Taylor expansion of terms in the evolution operator around $x=Q$ and apply it to the FGA ansatz to match the form in Lemma \ref{lemma:x-Q}. However, for the FSE, the kinetic term neither exhibits smoothness nor can be explicitly calculated. Specifically, we need to handle the terms of the form:
\begin{equation*}
    \kin(-\ri\veps \partial_x) \psi^{\veps}_{\FGA} 
    = \frac{1}{(2\pi\veps)^{5d/2}}\!\int_{\mathbb{R}^{3d}} \e^{\frac{\ri}{\veps} \Phi}\,\amp\,\varphi_0^{\veps} \dby \dbq \dbp \int_{\mathbb{R}^d} \kin(\xi + P) \e^{-\frac{1}{2\veps}(\xi - \ri(x-Q))^2} \dxi.
\end{equation*}
It is tedious and complicated to handle the singularity of $\kin(\xi)$ and extract the asymptotic expansion within the position space in terms of $x-Q$. To address this issue, noting that the kinetic term in the \FSE~is defined by a Fourier integral in the momentum space, it would be beneficial to handle its expansion and asymptotic orders in the momentum space. Therefore, we introduce the concepts of the dual phase function and its corresponding action function and quantify the equivalent asymptotic orders in terms of the momentum displacement $\xi-P$ instead of $x-Q$. This extension allows us to effectively manage the kinetic term in the \FSE.


The dual-phase function $\tilde{\Phi}$ is defined as
\begin{multline}\label{eq:tildePhi}
    \tilde\Phi(t,\xi,\ze,q,p) := \tilde{S}(t,q,p) - (\xi-P(t,q,p))\cdot Q(t,q,p)
    \\  + \f\ri2|\xi-P(t,q,p)|^2 + (\ze-p)\cdot q  + \f\ri2|\ze-p|^2 ,
\end{multline}
where the dual action $\tilde{S}$ is defined as the unique solution of
\begin{equation}\label{eq:tildeS}
    \ddt{\tilde{S}} = Q\cdot\p_QH(Q,P) - H(Q,P), \quad \tilde{S}(0,q,p) = 0 .
\end{equation}
We remark that the difference between $S$ and $\tilde{S}$ is
\begin{equation}\label{eq:diff-S}
    S-\tilde{S} = Q\cdot P - q\cdot p.
\end{equation}
Additionally, it is straightforward to verify that $\tilde{S}$ satisfies
\begin{equation}
    \p_q\tilde{S} + \p_qP\cdot Q=0 \quad\text{ and }\quad \p_p\tilde{S} + \p_pP\cdot Q - q=0,
\end{equation}
and we have the following relation,
\begin{equation}
    \p_z\tilde\Phi=\p_q\tilde\Phi-\ri\p_p\tilde\Phi = - (\p_zQ+\ri\p_zP)\cdot(\xi-P) = -Z(\xi-P).
\end{equation}
These relations are analogous to their counterparts in the original phase function and action, leading to a parallel asymptotic order analysis presented in the following Lemma \ref{lemma:xi-P} for the Fourier integral operator with dual phase function defined as
\begin{equation*}
    [\tFIO(u)\psi](\xi)=\frac{1}{(2\pi\veps)^{3d/2}} \int_{\mathbb{R}^{3d}} \e^{\ri\tilde{\Phi}(t,\xi,\zeta,q,p)/\veps} \,u(q,p)\,\psi(\zeta) \dze \dbq \dbp,
\end{equation*}
where $u\in \mathcal{S}(\mathbb{R}^{2d}; \mathbb{C})$ and $\psi \in\mathcal{S}(\mathbb{R}^d; \mathbb{C})$ are Schwartz-class functions.
\begin{lemma} \label{lemma:xi-P}
    For any $v(\ze, q, p)\in\bbR^{d}$
    in Schwartz class, we have
    \begin{align}
        \label{eq:v(xi-P)}
        (\xi-P)_j\,v_j \stackrel{\tilde\Phi}{\sim} &
        - \ri\veps \partial_{z_{k}}\left(v_{j} Z_{j k}^{-1}\right).
    \end{align}
\end{lemma}
For $(q,p)\in \mathbb{R}^{2d}$, define the Fourier-Bros-Iagolnitzer (FBI) transformations $\FBI_{\pm}$
on $\mathcal{S}(\mathbb{R}^d)$ as follows:
\begin{align}\label{eq:FBI}
    (\FBI_{-} f)(q,p) & = \frac{1}{2^{d/2}} \frac{1}{(\pi\varepsilon)^{3d/4}} \int_{\mathbb{R}^d} e^{-\frac{\mathrm{i}}{\varepsilon}p\cdot(x-q) -\frac{1}{2\varepsilon}|x-q|^2} f(x) \dbx,      \\
    (\FBI_{+} f)(q,p) & = \frac{1}{2^{d/2}} \frac{1}{(\pi\varepsilon)^{3d/4}} \int_{\mathbb{R}^d} e^{\frac{\mathrm{i}}{\varepsilon}q\cdot(\xi-p) -\frac{1}{2\varepsilon}|\xi-p|^2} f(\xi) \dxi.
\end{align}
One can easily verify that given function $u(x)$ and its Fourier transform $\hat{u}(\xi)$,
\begin{equation} \label{eq:equvi_FBI}
    (\FBI_{+} \hat{u}(\xi))(q, p) =  e^{-\frac{\ri}{\veps}q\cdot p} (\FBI_{-} u(x))(q, p).
\end{equation}
The pseudo inverse FBI transformations $(\FBI_{\pm})^{*}$ on $\mathcal{S}(\mathbb{R}^{2d})$ are given by
\begin{align}\label{eq:FBI_inv}
    ((\FBI_{-})^{*} g)(x)   & = \frac{1}{2^{d/2}} \frac{1}{(\pi\varepsilon)^{3d/4}} \int_{\mathbb{R}^{2d}} e^{\frac{\mathrm{i}}{\varepsilon}p\cdot(x-q) -\frac{1}{2\varepsilon}|x-q|^2} g(q,p) \dbq \dbp,      \\
    ((\FBI_{+})^{*} g)(\xi) & = \frac{1}{2^{d/2}} \frac{1}{(\pi\varepsilon)^{3d/4}} \int_{\mathbb{R}^{2d}} e^{-\frac{\mathrm{i}}{\varepsilon}q\cdot(\xi-p) -\frac{1}{2\varepsilon}|\xi-p|^2} g(q,p) \dbq \dbp.
\end{align}
We have the following property for the FBI transformation, whose proof can be found in \cite{martinez2002introduction}.
\begin{proposition} \label{prop:FBI_norm}
    For any $f\in \mathcal{S}(\mathbb{R}^d)$, $\| \FBI_{\pm} f \|_{L^2(\mathbb{R}^{2d})} = \| f \|_{L^2(\mathbb{R}^{d})}$.
    Hence the domain of $\FBI_{\pm}$ and $(\FBI_{\pm})^{*}$ can be extended to $L^2(\bbR^d)$ and $L^2(\bbR^{2d})$, respectively. Moreover, $(\FBI_{\pm})^{*}\FBI_{\pm} = \mathrm{Id}_{L^2(\bbR^{d})}$ but $\FBI_{\pm}(\FBI_{\pm})^{*} \neq \mathrm{Id}_{L^2(\bbR^{2d})}$.
\end{proposition}

\subsection{Formulation of the frozen Gaussian approximation}\label{sec:FGA-derivation}

The (1st-order) FGA is defined as
\begin{align} \label{eq:fga_standard}
    \psi^{\veps}_{\FGA}(t,x)
    = \left[ \FIO(\amp)\varphi_0^{\veps}\right](x)
    = \frac{1}{(2\pi\veps)^{3d/2}} \int_{\mathbb{R}^{3d}} \e^{\frac{\ri}{\veps} \Phi(t,x,y,q,p)}\,\amp(t,q,p)\,\varphi_0^{\veps}(y) \dby \dbq \dbp,
\end{align}
where $\varphi_0^{\veps}(x)$ is the initial condition, $\Phi$ is defined in \eqref{eq:Phi} with the {\it position center} $Q$ and the {\it momentum center} $P$ satisfying the Hamiltonian flow \eqref{eq:hamiltonian-flow}, the {\it action} $S$ satisfying \eqref{eq:action}, and the {\it amplitude} $a$ satisfying
\begin{align}\label{eq:amplitude}
    \ddt\amp = \frac{1}{2}  \text{tr}\left( Z^{-1} \ddt Z  \right) \amp, \qquad a(0,q,p) = 2^{d/2}.
\end{align}
\vspace{-15pt}
\begin{remark}
    The well-posedness of FGA system \eqref{eq:hamiltonian-flow}, \eqref{eq:action}, and \eqref{eq:amplitude} requires the Hamiltonian $H$ to be at least $C^2$, which is not valid for the \FSE. Here, we perform formal derivation by assuming $H$ is smooth enough, and leave the rigorous modification to the next section.
\end{remark}

To see how the FGA solution fits the Schr\"odinger equation, we substitute \eqref{eq:fga_standard} into the Schr{\"o}dinger equation \eqref{eq:Schrodinger} and compute each term.

For the time-derivative term,
\begin{align} \label{eq:psi_t}
    \ri\veps \partial_t \psi^{\veps}_{\FGA}
     & =\FIO \Big( \ri\veps\partial_t \amp - \amp\left((\partial_t S - P\cdot \partial_t Q) +  (x-Q)\cdot \partial_t(P-\ri Q)\right) \Big) \varphi_0^\veps
    \nonumber                                                                                                                                              \\&= \FIO  \left( \ri\veps\partial_t \amp - \amp(\partial_t S - P \cdot\partial_t Q) + \veps\partial_{z_k}(\partial_t(P - \ri Q)_j \amp Z^{-1}_{jk} ) \right) \varphi_0^\veps.
\end{align}

For potential function $V(x)$, using Taylor expansion of $V(x)$ around $x=Q$,
\begin{multline}
    V(x) = V(Q)  + \p_{Q_j} V(Q)\,(x-Q)_j  + \f12\,\p_{Q_{jk}}^2V(Q)\,(x-Q)^2_{jk} \\
    {} +  \f12(x-Q)^3_{jkl}\,\int_0^1(1-\tau)^2\p^3_{Q_{jkl}}V(Q+\tau(x-Q)) \rd\tau.
\end{multline}
Then using Definition \ref{def:sim} and Lemma \ref{lemma:x-Q},
\begin{multline}\label{eq:vpsi} 
    V(x)\psi^{\veps}_{\FGA} = \FIO
    \Big(~
    V(Q)\amp -\veps\partial_{z_k}\left(\partial_{Q_j} V(Q) \amp Z^{-1}_{jk}\right) \\
    + \frac{\veps}{2}\,\partial_{z_l}Q_j\,\partial^2_{Q_{jk}} V(Q)\, \amp\, Z^{-1}_{kl}
    + \veps^2{R_V(x,q,p)}
    ~\Big) \psieo ,
\end{multline}
where $R_{V}(x,q,p)$ denotes the remainder term to be discussed latter.

For the kinetic operator, it is useful to consider the FGA solution in the momentum space. Taking Fourier transform of $\psie_\FGA$, we get
\begin{equation}\label{eq:fga_fourier}
    \hpsie_\FGA(t,\xi) = \left[\tFIO(\amp)\hat\varphi_0^\veps\right](\xi) = \frac{1}{(2\pi\veps)^{3d/2}} \int_{\mathbb{R}^{3d}} \e^{\frac{\ri}{\veps} \tilde\Phi(t,\xi,\ze,q,p)}  \,\amp(t,q,p)\,\hat\varphi_0^\veps(\ze) \dze \dbq \dbp,
\end{equation}
where we have used the equality \eqref{eq:diff-S} and $\tilde\Phi$ is defined in \eqref{eq:tildePhi}.
By \eqref{eq:kin-operator},
$$\kin(-\ri\veps\p_x)\psie_\FGA = \IFT\left[\kin(\xi)\hpsie_{\FGA}\right].$$
Thus, we expand
\begin{multline}\label{eq:kin_taylor}
    \kin(\xi) = \kin(P) + \p_{P_j}\kin(P)\,(\xi-P)_j + \f12\p_{P_{jk}}^2\kin(P)\,(\xi-P)^2_{jk}
    \\{} +
    \f12(\xi-P)^3_{jkl}\,\int_0^1(1-\tau)^2\p^3_{P_{jkl}}\kin(P+\tau(\xi-P)) \, \rd\tau .
\end{multline}
Therefore by Lemma \ref{lemma:xi-P} we can write
\begin{multline}\label{eq:kinpsi-fourier} 
    \kin(\xi)\hpsie_{\FGA} = \tFIO
    \Big(~
    \kin(P)\amp -\ri\veps\partial_{z_k}\left(\partial_{P_j} \kin(P) \amp Z^{-1}_{jk}\right)
    +\frac{\ri\veps}{2}\,\partial_{z_l}P_j\,\partial^2_{P_{jk}} \kin(P)\, \amp\, Z^{-1}_{kl}
    \\+ \veps^2{R_\kin(\xi,q,p)}
    ~\Big) \hat\varphi_0^\veps ,
\end{multline}
where $R_{\kin}(\xi,q,p)$ denotes the remainder term.
Remark that the $R_V$ and $R_\kin$ are of similar form, except that $R_V$ is obtained from expanding the potential $V$ in position space, while $R_\kin$ is obtained from expanding the kinetic energy $\kin$ in the momentum space.
Take $R_\kin$ for example. In deriving \eqref{eq:kinpsi-fourier} from $\eqref{eq:kin_taylor}$ by Lemma \ref{lemma:xi-P}, the 2nd-order derivative term gives
\begin{multline}\label{eq:d2T-sim}
    \amp\, \p_{P_{jk}}^2\kin(P)\,(\xi-P)^2_{jk}
    \stackrel{\tilde\Phi}{\sim}  ~\ri\veps \,\amp\, \partial_{z_l} P_{j}\, \p_{P_{jk}}^2\kin(P)\, Z_{k l}^{-1} \\ - \veps^{2}\,\partial_{z_{m}}\left(\partial_{z_{l}}\left(\amp\, \p_{P_{jk}}^2\kin(P) \, Z_{k l}^{-1}\right) Z_{j m}^{-1}\right),
\end{multline}
and by denoting $\dkin3_{jkl}=3\int_0^1(1-\tau)^2\p^3_{P_{jkl}}\kin(P+\tau(\xi-P)) \, \rd\tau$, the 3rd-order derivative term gives
\begin{align}
    \amp\, (\xi-P)^3_{jkl}\,\dkin3_{jkl}
    \stackrel{\tilde\Phi}{\sim} & ~ {}\veps^2\,\p_{z_n}P_l\,\partial_{z_m}\left(\amp\,\dkin3_{jkl}Z_{jm}^{-1}\right)\,Z_{kn}^{-1}
    +2\veps^2 \partial_{z_n}\left(\amp\,\partial_{z_m}P_l\,\dkin3_{jkl}Z_{jm}^{-1}Z_{kn}^{-1}\right)
    \nonumber                                                                                                                     \\
                                & ~~~~~~~~~~~~~~~~~~~~~~~~~~
    {} + \ri\veps^{3} \p_{z_r} \left( \partial_{z_{n}}\left(\partial_{z_{m}}\left(\amp\,\dkin3_{j k l} Z_{j m}^{-1}\right) Z_{k n}^{-1}\right)  Z^{-1}_{lr} \right).
    \label{eq:d3T-sim}
\end{align}
Then $\veps^2R_\kin$ in \eqref{eq:RT} consists all the $O(\veps^2)$ and $O(\veps^3)$ terms in \eqref{eq:d2T-sim} and \eqref{eq:d3T-sim}, that is
\begin{align}
    R_\kin(\xi, q, p) = {}-\f12 \partial_{z_{m}}\left(\partial_{z_{l}}\left(\amp\,\p^2_{P_{jk}}\kin Z_{k l}^{-1}\right) Z_{j m}^{-1}\right)
    -\f16\,\p_{z_n}P_l\,\partial_{z_m}\left(\amp\,\dkin3_{jkl}Z_{jm}^{-1}\right)\,Z_{kn}^{-1}
    ~~~\nonumber \\
    {}+\f13 \partial_{z_n}\left(\amp\,\partial_{z_m}P_l\,\dkin3_{jkl}\,Z_{jm}^{-1}\,Z_{kn}^{-1}\right)
    {}+\f\ri6\veps\p_{z_r} \left( \partial_{z_{n}}\left(\partial_{z_{m}}\left(\amp\,\dkin3_{j k l} Z_{j m}^{-1}\right) Z_{k n}^{-1}\right)  Z^{-1}_{lr} \right).
    \label{eq:RT}
\end{align}

Geting together\eqref{eq:psi_t}, \eqref{eq:vpsi}, and  \eqref{eq:kinpsi-fourier}, and using the fact that the Hamiltonian flow satisfies $\partial_t(P - \ri Q)_j  + \ri\,\p_{P_j}\kin(P) + \p_{Q_j}V(Q) = 0$, we obtain
\begin{multline}\label{eq:SchrFGA}
    \Big(\ri\veps\p_t-\kin(-\ri\veps\p_x)-V(x)\Big)\psie_\FGA
    \\ = \FIO
    \Big[
    -\amp(\partial_t S - P \cdot\partial_t Q + \kin(P) + V(Q) )
    {} + \ri\veps\partial_t \amp
    - \frac{\ri\veps}{2}\partial_{z_l} P_{j} \, \p_{P_{jk}}^2\kin \, Z_{k l}^{-1}  \, \amp
    \\
    ~~{} -  \frac{\veps}{2}\,\partial_{z_l}Q_j\,\partial^2_{Q_{jk}} V\, Z^{-1}_{kl}\, \amp
    - \veps^2 R_V(x,q,p)
    \Big]\varphi_0^\veps
    - \veps^2 \IFT\left[ \tFIO \left(R_\kin(\xi,q,p)\right)\hat\varphi_0^\veps \right].
\end{multline}
We define the following operators:
\begin{align}
    \mathcal{L}_0 \amp & :=-(\partial_t S - P \cdot\partial_t Q + \kin(P) + V(Q) ) \amp,
    \\\mathcal{L}_1 \amp&:=\ri\p_t \amp - \frac{\ri}{2}\partial_{z_l} P_{j} \, \p_{P_{jk}}^2\kin \, Z_{k l}^{-1}  \, \amp - \frac{1}{2} \p_{z_l}Q_j \, \p_{Q_{jk}}^2 V \, Z_{k l}^{-1} \, \amp,
    \\ \mathcal{R}a &:= -\FIO \left(R_V(x,q,p)\right)\varphi_0^\veps  -  \IFT\left[ \tFIO \left(R_\kin(\xi,q,p)\right)\hat\varphi_0^\veps \right]. \label{eq:reminder}
\end{align}
Thus, \eqref{eq:SchrFGA} can be shortened to
\begin{equation}\label{eq:SchrFGA3}
    \Big(\ri\veps\p_t-\kin(-\ri\veps\p_x)-V(x)\Big)\psie_\FGA
    = \FIO
    \Big[~\mathcal{L}_0 \amp + \veps\,\mathcal{L}_1 \amp ~\Big]\varphi_0^\veps + \veps^2\mathcal{R}\,\amp.
\end{equation}
The $\mathcal{L}_0$ term is zero by \eqref{eq:hamiltonian-flow} and \eqref{eq:action}. For $\mathcal{L}_1$, noting the fact that
\begin{align}
    \ddt{}\left({\p_{z_l}Q_k}\right) = \p_{z_l}\left(\ddt{Q_k}\right) = \p_{z_l}\left(\p_{P_k}\kin(P)\right) = \partial_{z_l} P_{j} \, \p_{P_{jk}}^2\kin(P),
    \\
    \ddt{}\left({\p_{z_l}P_k}\right) = \p_{z_l}\left(\ddt{P_k}\right) = \p_{z_l}\left(-\p_{Q_k}V(Q)\right) = -\partial_{z_l} Q_{j} \, \p_{Q_{jk}}^2 V(Q),
\end{align}
we obtain
\begin{equation}
    \mathcal{L}_1 a
    = \ri\p_t\amp - \f\ri2 \, \text{tr}\left( \ddt{Z}Z^{-1} \right)\amp = 0,
\end{equation}
where the last equation is implied by the evolutionary equation \eqref{eq:amplitude} for the amplitude. Thus, only the $O(\veps^2)$ terms remain in \eqref{eq:SchrFGA3}, and we have
\begin{equation}\label{eq:SchrFGA-4}
    \Big(\ri\veps\p_t-\kin(-\ri\veps\p_x)-V(x)\Big)\psie_\FGA  = \veps^2\,\mathcal{R}\,\amp.
\end{equation}

\section{The FGA for the modified \FSE}
We have shown that the governing equation \eqref{eq:SchrFGA-4} for the FGA solution is formally an $O(\veps^2)$ perturbation to the Schr\"odinger equation \eqref{eq:Schrodinger} if $\mathcal{R}\,\amp$ is bounded.
The remainder $\mathcal{R}\,\amp$ arises from the Taylor expansion of the Hamiltonian symbol. Therefore, it can be bounded in cases where the Hamiltonian is sufficiently smooth, \eg, if $H(q,p)\in C^{\infty}(\bbR^{2d};\bbR)$. However, in cases where the Hamiltonian may have singularities, such as in the \FSE,~$H(q,p)=\f{|p|^\alpha}\alpha+V(q)$ with $1\leq\alpha<2$, the high-order derivatives of $H$ with respect to $p$ are singular at $p=0$.

Notice that $\mathcal{R}\,\amp$ in \eqref{eq:reminder} consists of two parts: $R_V$ and $R_\kin$. Since the Fourier transform is unitary in $L^2$ norm and we have assumed that the initial $\varphi_0^\veps$ is in Schwartz class, the analysis of the two parts of $\mathcal{R}\,\amp$ follow the same approach.
Let us assume that $V\in C^{\infty}$ and consider $\kin$ may produce singularities at $p=0$.
Recall that the expression of $R_\kin$ \eqref{eq:RT} involves at least the $4$th-order and at most the $6$th-order derivatives of the kinetic symbol $T$. The singularities in high-order derivatives make the boundedness of $R_\kin$ challenging. Additionally, we need to clarify how to evolve the ODEs \eqref{eq:hamiltonian-flow} and \eqref{eq:amplitude} when the trajectory touches the singularity at $|P|=0$.

To overcome these difficulties, we introduce a singularity-removed kinetic symbol by replacing $|p|$ with $\sqrt{|p|^2+\delta^2}$ and define
$\kindt(p) := \kin(\sqrt{|p|^2+\delta^2})$.
We consider $\psiedt$ as the solution of the modified Schr\"odinger equation 
\begin{equation}\label{eq:Schrodinger-delta}
    \ri\veps \partial_t \psiedt = \kindt(-\ri\veps\partial_x)  \psiedt + V(x) \psiedt,
\end{equation}
with the same initial condition $\psiedt(0,x)=\varphi_0^\veps(x)$. It can be verified that when $\delta > 0$ is sufficiently small, $\psiedt$ closely approximates the solution of the original Schr\"odinger equation, $\psie$. The proof of this assertion will be presented later.

While the modification $\sqrt{|p|^2 + \delta^2}$ addresses the evolution of ODEs in FGA when encountering singularities at $|P|=0$, there might still be singular points of the Hamiltonian equation \eqref{eq:hamiltonian-flow} that possess constant solution with $|P(t)|\equiv0$. To address this issue, we introduce a cutoff with parameter $\omega$ for the initial condition of the Hamiltonian system to remove these singular points and ensure that no such constant solutions arise throughout the evolution.

In the following, we present a series of lemmas and propositions related to the construction of the approximation chain,
$\psie \longrightarrow \psie_{\delta} \longrightarrow \psie_{\delta, \text{FGA}} \longrightarrow \psie_{\delta, \text{FGA}, \omega}$.
These results will be combined to establish the convergence in the proof of our main theorem.

\subsection{Cutoff strategy}

Firstly we consider the cutoff to remove those singular points of the Hamiltonian system with constant solution $P(t)\equiv 0$. For $\omega > 0$, we define an open set
including singular points of Hamiltonian equations \eqref{eq:hamiltonian-flow} as
$$
    B_{\omega} = \left\{ (q, p)\in\bbR^{2d} \,\Big|\, |p|^2 + |q-q_0|^2 < \omega^2, \, \forall q_0 \in \bbR^{d} \text{ such that } \partial_q V(q_0) = 0 \right\}.
$$
Then we define the closed bounded set $K_{\omega} \subseteq \bbR^{2d}$ excludes those singular points as
\begin{equation} \label{eq:K_omega}
    K_{\omega} = \left\{ (q,p)\in \bbR^{2d} \setminus B_{\omega} \,\Big|\, |q|^2 + |p|^2 \leq {1}/{\omega^2}  \right\}.
\end{equation}
For the Hamiltonian system over closed set $K_{\omega}$, we have the following Lemma.
\begin{proposition} \label{lemma:motion_lower_bound}
    Given $\omega > 0$, $\Tf>0$, if $Q(t,q,p)$, $P(t,q,p)$ satisfy Hamiltonian equations \eqref{eq:hamiltonian-flow}, then
    \begin{equation}
        |P(t,q,p)| + |\partial_Q V(Q(t,q,p))| \geq C_{\omega, \Tf} > 0,
    \end{equation}
    for $t\in[0, \Tf]$, $(q, p)\in K_{\omega}$, where $C_{\omega, \Tf}$ is a constant.
\end{proposition}

\begin{proof}
    Denote one of the singular points of \eqref{eq:hamiltonian-flow} as $(q_0, p_0)$, then $|p_0|=0$, $|\partial_q V(q_0)|=0$. Consider closed set $\Omega_t = \left\{ (Q(t,q,p), P(t,q,p)) \,|\, (q, p)\in K_{\omega} \right\}$. Noting that a trajectory starts from $(q_0,p_0)$ will be a constant solution and all trajectories will not intersect during propagation, there will always be a positive distance between closed set $\Omega_t$ and singular point $(q_0, p_0)$. Thus,
    \begin{equation}
        |P(t,q,p)| + |\partial_Q V(Q(t,q,p))| > 0, \quad (t,q,p) \in [0, \Tf] \times K_{\omega}.
    \end{equation}
    Since $|P(t,q,p)| + |\partial_Q V(Q(t,q,p))|$ is a continuous function over the compact set $[0, \Tf] \times K_{\omega}$, it will reach a minimum $C_{\omega, \Tf} > 0$ over $[0, \Tf] \times K_{\omega}$, which completes the proof.
\end{proof}

\begin{remark}
    Prop. \ref{lemma:motion_lower_bound} primarily describes the phenomena that for a conserved system, if initial states are away from the equilibrium positions (singular points of Hamiltonian equations) where momentum $p$ and force $\partial_q V(q)$ are zero at the same time, then the trajectory will not be stuck in a fixed point during propagation.
\end{remark}

Motivated by the WKB function that is typical in the study of high-frequency wave propagation, we make the following assumption regarding the initial wave-function:
\begin{assumption}\label{asym-high-freq}
    The initial wave-function $\psie(t=0)=\psieo\in L^2(\bbR^d)$ satisfies $\norm{\psieo}_{L^2}=1$. Furthermove, $\psieo$ is an asymptoticallly high-frequency function, as defined in Definition~\ref{def:ahf}. Specifically, we assume that there exist functions $S_0\in C^{\infty}(\bbR^d)$ and $\rho_0\in \mathcal{S}(\bbR^d)$ such that $\rho_0\geq 0$, $\norm{\rho_0}_{L^1}\leq2\norm{\psieo}_{L^2}$, and
    \begin{equation*}
        \begin{aligned}
            \abs{(\FBI_-\psieo)(q,p)}^2 \leq  \f{\rho_0(q)}{(2\pi\veps)^{d/2}} \, \e^{-\f{\abs{p-\p_qS_0(q)}^2}{2\veps}}.
        \end{aligned}
    \end{equation*}
\end{assumption}
\begin{definition}[Asymptotically High-Frequency Function]\label{def:ahf}
    Let $\{u^{\veps}\} \subseteq L^2(\bbR^d)$ be a family of functions such that $\left\| u^{\veps} \right\|_{L^2(\bbR^d)}$ is uniformly bounded. Given $\omega > 0$, we say that $\{u^{\veps}\}$ is asymptotically high-frequency with cutoff $\omega$ if
    \begin{equation}
        \int_{\bbR^{2d}\setminus K_{\omega}} \left| (\FBI_{-} u^{\veps})(q,p) \right|^2 \dbq \dbp = O(\veps^{\infty})
    \end{equation}
    as $\veps \rightarrow 0$. The notation $A^{\veps} = O(\veps^{\infty})$ means that for any $k\in\N$,
    $
        \lim_{\veps\rightarrow 0} \veps^{-k} |A^{\veps}|=0
    $ .
\end{definition}

\begin{remark}
    Although in the definition above we use two constraints for $u^{\veps}$ and $\hat{u}^{\veps}$ respectively, they are actually equivalent by \eqref{eq:equvi_FBI}, and only one of them need to be checked when verifying whether a function is asymptotically high-frequency or not.
\end{remark}

\begin{remark}
    The closed bounded set $K_{\omega}$ defined in \eqref{eq:K_omega} over phase-space $\bbR^{2d}$ identifies where the initial wave function is primarily supported under the FBI transform. It also excludes the singular points of Hamiltonian flow \eqref{eq:hamiltonian-flow} by removing small open balls around these points. That is why we introduce the cutoff FGA.
\end{remark}

Let $\chi_{\omega}: \bbR^{2d} \rightarrow [0, 1]$ be a smooth cutoff function with $\omega > 0$ such that
$\chi_{\omega}=1$ on $K_{\omega}$ and $\chi_{\omega}=0$ on $\bbR^{2d} \setminus K_{\omega/2}$.
Then, for any $k\in \mathbb{N}$, there exists constant $C_{k, \omega} > 0$ such that
\begin{equation}
    \sup_{(q,p)\in\bbR^{2d}} \sup_{|\alpha|=k} \left| \partial_q^{\alpha_q}\partial_p^{\alpha_p} \chi_{\omega}(q,p) \right| \leq C_{k, \omega}.
\end{equation}
Now we can define the frozen Gaussian approximation with cutoff $\omega$ as
\begin{equation}
    \psi_{\text{FGA}, \omega}^{\veps}(t,x)
    = \frac{1}{(2\pi\veps)^{3d/2}} \int_{\mathbb{R}^{3d}} \e^{\frac{\ri}{\veps} \Phi(t,x,y,q,p)}\,\amp(t,q,p)\,\chi_{\omega}(q,p)\,\varphi_0^{\veps}(y) \dby \dbq \dbp.
\end{equation}
To simplify notations, define a filtered amplitude $\tilde{\amp}(t,q,p):=\amp(t,q,p)\chi_{\omega}(q,p)$ satisfying
\begin{equation} \label{eq:filtered_amplitude}
    \ddt {\tilde{\amp}} = \frac{1}{2} \text{tr}\left( \ddt Z Z^{-1} \right) \tilde{\amp}, \qquad \tilde{\amp}(0,q,p) = 2^{d/2} \chi_{\omega}(q,p).
\end{equation}
From the definition of filtered amplitude, we know that the value of $\amp(t,q,p)$ outside $\text{supp}\, \chi_{\omega}(q,p) \subseteq K_{\omega/2}$ will not affect $\tilde{\amp}(t,q,p)$.
Since $\amp(t,q,p)$ and $\tilde{\amp}(t,q,p)$ satisfy the same equation with different initial data, the original FGA $\psi^{\veps}_{\text{FGA}}$ and cutoff FGA $\psi^{\veps}_{\text{FGA}, \omega}$ also satisfy the same differential equation \eqref{eq:SchrFGA-4} except that they have difference initial values. Specifically, when $t=0$,
\begin{align*}
    \psi^{\veps}_{\text{FGA}}(0, x)= (\FBI_{-})^*(\FBI_{-} \varphi_0^{\veps})=\varphi_0^{\veps}(x),  \quad\text{while }\quad
    \psi^{\veps}_{\text{FGA}, \omega}(0, x) = (\FBI_{-})^*(\chi_{\omega} \FBI_{-} \varphi_0^{\veps}) .
\end{align*}
Despite the slight difference in their initial values, these variations do not impact the convergence of the FGA when asymptotically high-frequency functions are considered as initial conditions. This cutoff strategy has been thoroughly discussed in the context of the convergence of conventional FGAs; see \eg \cite{LuYa:CPAM}. Therefore, in the following discussion, we will not distinguish between $\psi^{\veps}_{\text{FGA}}$ and its cutoff version $\psi^{\veps}_{\text{FGA}, \omega}$, opting instead to use $\psi^{\veps}_{\text{FGA}}$ for notational simplicity.


\subsection{The modified Hamiltonian flow}

Now we move forward to discuss the estimations of the Hamiltonian flow with the modified kinetic symbol $\kindt$. We make the following assumption for the system \eqref{eq:hamiltonian-flow}, which will be assumed for the rest of the paper without further indication.

\begin{assumption}\label{asp:global_lipschitz}
    The potential function is $C^{\infty}$ function. Specifically, there exists a constant $C>0$ such that $ \|\partial_Q V(Q) \|_{L^{\infty}(\bbR^d)} \leq C$.
\end{assumption}


Despite the presence of singularities in the higher-order derivatives of $H(Q, P)$ at $P=0$, the solution of the Hamiltonian flow \eqref{eq:hamiltonian-flow} remains well-defined and unique. We conclude this property as the proposition stated below.

\begin{proposition} \label{prop:unique_solution}
    Consider the Hamiltonian $H$ in \eqref{eq:hamiltonian}
    with $1<\alpha\leq 2$ and smooth potential $V(Q)$, then the solution to the Hamiltonian flow \eqref{eq:hamiltonian-flow}
    is unique.
\end{proposition}

The proof is straightforward when employing Picard successive approximation; therefore, we omit the details here.

\begin{corollary} \label{cor:continue_delta}
    Perturbed momentum center $P(t)$ and position center $Q(t)$ satisfy \eqref{eq:hamiltonian-flow} with modified kinetic symbol $\kindt(|P|)$ are continuous with respect to parameter $\delta$.
\end{corollary}

\begin{proof}
    The proof can be done in the same way as Proposition \ref{prop:unique_solution}.
\end{proof}


Additionally, the cutoff strategy can be applied to the modified Hamiltonian flow while maintaining the same properties as Prop. \ref{lemma:motion_lower_bound}:
\begin{proposition} \label{prop:lower_bound_dV}
    For dynamic system $H^{\delta}(Q, P)=\kindt(P) + V(Q)$ with $\delta>0$ small enough, given $\omega > 0$, $\Tf>0$ we have
    \begin{equation}
        |P(t,q,p)| + |\partial_{Q} V(Q(t,q,p))| \geq C_{\omega, \Tf} > 0,
    \end{equation}
    for $t\in[0, \Tf]$, $(q, p)\in K_{\omega}$, where $C_{\omega, \Tf} > 0$ is a constant independent of $\delta$.
\end{proposition}


\begin{remark}
    The introduction of the singularity-removed kinetic symbol $\kindt$ leads to some new estimations relative to $\delta$ for most quantities involved. The only difference between the FGA formulations for \eqref{eq:Schrodinger-delta} and \eqref{eq:Schrodinger} is the kinetic symbol. To avoid confusion, we do not introduce new notation but continue to use $\amp$, $S$, $Q$, and $P$ for the amplitude, action, position center, and momentum center, referring to the same equations in Section \ref{sec:FGA} to construct $\psie_\FGAdt$.
\end{remark}



Before delving into the convergence analysis of the FGA, we require the following estimates for the Hamiltonian flow and its associated auxiliary quantities. First, we present two propositions to estimate $Q(t)$ and $P(t)$. Subsequently, we provide $\delta$-dependent estimates for their higher-order derivatives in a compact form using the associated canonical transformation.

\begin{proposition} \label{prop:estimate_P}
    For the Hamiltonian flow \eqref{eq:hamiltonian-flow} with the modified kinetic symbol $\kindt(P)$, given initial conditions $(q,p)\in K_{\omega}$ and an evolution time interval $[0, \Tf]$, when $P(t)$ passes through its zero point at $t=t_0$, there exists a small time interval, denoted as $[t_0, t_1]$, whose length depends only on $\omega$ and $\Tf$, such that for $t\in [t_0, \min(t_1, \Tf)]$, we have
    \begin{equation}
        |P(t)| \geq C_{\omega, \Tf} \, (t-t_0),
    \end{equation}
    where $C_{\omega, \Tf} > 0$ is a constant. Therefore, there are only a finite number of zero points of $|P(t)|$ within a given finite evolution time interval.
\end{proposition}

The key to the proof is to observe that for $t\in [t_0, t_1]$, due to the continuity of the Hamiltonian flow, one can derive $|P(t)| \geq \frac{1}{2} |\partial_Q V(Q(t_0))| \, (t-t_0)$. Subsequently, the result follows from Proposition \ref{prop:lower_bound_dV}.


\begin{proposition} \label{prop:bound_PQ}
    Given $\Tf>0$, then for any $t\in [0, \Tf]$ and $(q, p) \in K_\omega$, we have
    \begin{equation} \label{eq:linear_bound_P}
        |p|-Ct \leq |P(t)| \leq |p|+Ct.
    \end{equation}
    Furthermore,
    $
        |P(t, q, p)| \leq M + C\Tf$, and $|Q(t, q, p)| \leq M + C_1(M + C\Tf)^{\alpha-1}\Tf
    $,
    where $M,\, C,\, C_1 > 0$ are constants.
\end{proposition}

\begin{proof}
    Differentiating $|P|$ with respect to time variable $t$, from \eqref{eq:hamiltonian-flow} we have
    $$
        \ddt{}|P| = \frac{1}{|P|} P \cdot \ddt{P} = - \frac{P}{|P|} \cdot \partial_Q H(Q,P).
    $$
    By Assumption \ref{asp:global_lipschitz},
    $$
        \ddt{}|P| \leq |\frac{P}{|P|}|\,|\partial_Q V(Q)| \leq C.
    $$
    Integrating with $t$ from $0$ to $\tau>0$ at both sides, by the initial conditions we have
    $$
        |P(\tau)| \leq |p| + C\tau.
    $$
    For another, differentiating $|P|^{-1}$ gives
    $$
        \ddt{}(\frac{1}{|P|})=-\frac{1}{|P|^3} P \cdot \partial_Q H(Q,P),
    $$
    therefore,
    $$
        |P|^2\ddt{}(\frac{1}{|P|}) = -\frac{1}{|P|} P \cdot \partial_Q H(Q,P) \leq  \frac{1}{|P|} |P \cdot \partial_Q H(Q,P)| \leq C. 
    $$
    Integrating with $t$ from $0$ to $\tau>0$ we have
    $
        |P(\tau)| \geq |p|-C\tau.
    $
    Since $(q, p)\in K_\omega$ and $t\in [0,\Tf]$, we obtain $|P(t, q, p)| \leq M + C\,\Tf$.

    In the same way, differentiating $|Q|$ with $t$ we have
    $$
        \ddt{}|Q| = \frac{1}{|Q|} Q \cdot \partial_P H \leq C_1 |P|^{\alpha-1}.
    $$
    Then integrate with time variable $t$ at both sides, note that $(q, p)\in K_\omega$ and $|P| \leq M + C\,\Tf$, we obtain $|Q| \leq M + C_1(M + C\,\Tf)^{\alpha-1}\,\Tf$.
\end{proof}

For later discussion, we introduce a notation for higher order derivatives with respect to $(q,p)$ as follows. For $u\in C^{\infty}(\mathbb{R}^{2d}, \mathbb{C})$ and $k\in\mathbb{N}$, define
\begin{equation}
    \Lambda_{k, \omega}[u] = \max_{|\beta_q|+|\beta_p|\leq k} \sup_{(q,p)\in K_\omega} |\partial_q^{\beta_q}\partial_p^{\beta_p}u(q, p)|,
\end{equation}
where $\beta_q$ and $\beta_p$ are multi-indices corresponding to $q$ and $p$, respectively.

The convergence analysis of the FGA  relies on the boundedness of the derivatives of the auxiliary functions associated with the Hamiltonian flow. A central result is encapsulated in the following proposition:
\begin{proposition} \label{prop:bound_F}
    Given $\Tf>0$, there exists a constant $C$, independent of $\delta$, such that
    \begin{align}  \label{eq:bound_F}
        \sup_{t\in [0, \Tf]} \Lambda_{0,\omega} \left[F(t)\right] < C \;\;\text{ and }\;
        \sup_{t\in [0, \Tf]} \Lambda_{1,\omega} \left[F(t)\right] < C\, \left(\abs{P(t)}^2+\delta^2\right)^{\f{\alpha-2}2}. 
    \end{align}
\end{proposition}
We note that for conventional FGAs, the uniform boundedness of $\Lambda_{k,\omega}\left[F(t)\right],\,k\geq0$ can be directly established using Gr\"onwall's inequalities, provided the coefficients of the PDE are sufficiently smooth \cite{LuYa:CPAM}. However, in the case of the \FSE, the singularity of the kinetic symbol introduces $\delta$-dependent estimates in \eqref{eq:bound_F}. A key observation is that the $\delta$-dependent singularity occurs only locally near $P(t)=0$. This localization allows us to control the remainder term in the PDE through time integration, ultimately ensuring convergence as demonstrated in Section \ref{sec:convergence}. The detailed proof of this proposition is deferred to Appendix \ref{app:bounds-F}, and here we will only present some of its direct corollaries.
\begin{corollary} \label{cor:bound_Z}  \label{prop:bound_Z} \label{prop:bound_a}
    Given $\Tf>0$, for the function $g(t)$ being any of $Z$, $Z^{-1}$, or $\tilde\amp$, there exist a constant $C>0$ such that
    \begin{align} \label{eq:bound_Z}
        \sup_{t\in [0, \Tf]} \Lambda_{0,\omega} \left[g(t)\right] < C \;\;\text{ and }\;
        \sup_{t\in [0, \Tf]} \Lambda_{1,\omega} \left[g(t)\right] < C\, \left(\abs{P(t)}^2+\delta^2\right)^{\f{\alpha-2}2}.
    \end{align}
\end{corollary}
\begin{proof}
    These results follow essentially the same proof in \cite[for Lemma 5.1--5.3]{LuYa:CPAM}, expect that the given bounds for $\Lambda_{1,\omega} \left[F(t)\right]$ are now $\delta$-dependent. One can show that $Z$ is invertible, and furthermore, $Z$, $Z^{-1}$, and $\tilde\amp$ are in the same order as $F$, and so do their corresponding partial derivatives. We omit the details of the proof here.
\end{proof}

\subsection{The modified FGA and the associated remainder}\label{sec:modified-remainder}
To define the FGA solution $\psie_\FGAdt$ for the modified Schrödinger equation \eqref{eq:Schrodinger-delta}, we proceed by following similar derivation steps used to obtain equation \eqref{eq:SchrFGA-4} in Section \ref{sec:FGA-derivation}. This leads us to:
\begin{equation}\label{eq:SchrFGA-delta}
    \Big(\ri\veps\p_t-\kindt(-\ri\veps\p_x) - V(x)\Big)\psie_\FGAdt  = \veps^2\,\mathcal{R}^\delta\,\amp.
\end{equation}
Here, $\mathcal{R}^\delta$ mirrors the form of $\mathcal{R}$ from equation \eqref{eq:reminder}, with the kinetic symbol $\kin$ replaced by $\kindt$. However, as indicated in equation \eqref{eq:RT}, this formulation involves up to 3rd-order derivatives of auxiliary quantities, such as $\partial_z^3\amp$ and $\partial_z^3Z^{-1}$.
Despite this requirement, our current estimates provided in Proposition \ref{prop:bound_F} and Corollary \ref{prop:bound_Z} are limited to derivatives of order 0 and 1. Additionally,  an $O(\delta^{\alpha-2})$ blow-up may occur when $P(t)=0$,  primarily due to the insufficient smoothness of the kinetic symbol $\kin$.
To address this challenge, it is necessary to carefully revisit the asymptotic analysis in \eqref{eq:kin_taylor} and \eqref{eq:kinpsi-fourier}. Note that in obtaining \eqref{eq:RT}, we applied Lemma \ref{lemma:xi-P} for multi-times. To reduce the reliance on higher-order derivatives, we ``roll back" for $(\xi-P)^2$ term and keep the $(\xi-P)^3$ without applying Lemma \ref{lemma:xi-P} . This yields:
\begin{align}\label{eq:RT-roll-back}
    R_T^\delta & = \underbrace{ - \f\ri{2\veps} (\xi-P)_{j} \,\partial_{z_{l}}\!\left(\amp\, \p_{P_{jk}}^2\kindt(P) \,Z_{k l}^{-1}\right) }_{=:R_1}
    \,+\, \underbrace{ \f1{6\veps^2} (\xi-P)^3_{jkl}\,\amp\,\dkindt3_{jkl}(P,\xi-P) }_{=:R_2} ,
\end{align}
where $\displaystyle\dkindt3_{jkl}(P,\xi-P)=3\int_0^1(1-\tau)^2\,\p^3_{P_{jkl}}\kindt(P+\tau(\xi-P)) \, \rd\tau$.
For simplicity of notation, we will omit the sub-indices in the remainder terms and instead adopt a one-dimensional representation.

\subsubsection{Estimate for \texorpdfstring{$R_1$}{R1}}

To estimate $\norm{\tFIO\left( R_1 \right)\hpsieo}_{L^2}$, for $v\in L^2(\bbR_\xi^d)$ we compute
\begin{equation}
    \begin{aligned}
        \langle v, \tFIO\left( R_1 \right)\hpsieo \rangle
         & =
        \frac{-\ri}{(2\pi\veps)^{3d/2}} \int_{\bbR^{4d}} \overline{v}(\xi)\, \e^{\ri \tilde{\Phi}(t,\xi,\zeta,q,p)/\veps} \,\f{\xi-P}{2\veps}\,
        \\& ~~~~~~~~~~~~~~~~~~~~~~~~~~~{}\times\partial_{z}\left(\amp\, \p_{P}^2\kindt(P) \,Z^{-1}\right)\, \hpsieo(\zeta)\, \dze \dbq \dbp \dxi
        \\& = \int_{\bbR^{2d}} \dbq \dbp \;  \tilde{v}(Q,P)
        \,\partial_{z}\!\left(\amp\, \p_{P}^2\kindt(P) \,Z^{-1}\right) \, \e^{\ri S/ \veps} \, (\FBI \hpsieo)(q,p) ,
        \label{eq:R1-Holder}
    \end{aligned}
\end{equation}
where for the last equality we have defined
\begin{align*}
    \tilde{v}(Q,P) & := \frac{1}{2^{d}(\pi\veps)^{3d/4}}\int_{\bbR^d} \f{\xi}{\veps} \, \e^{-\frac{\ri}{\veps} Q\cdot\xi - \frac{1}{2\veps}\abs\xi^2} \, \overline{v}(\xi+P) \dxi.
\end{align*}
To use H\"older's inequalities for \eqref{eq:R1-Holder}, on one hand, we compute
\begin{equation}
    \begin{aligned}
        \norm{\tilde{v}}_{L^2(\bbR^{2d})}^2
        = & ~ \frac{1}{2^{2d}(\pi\veps)^{3d/2}}\int_{\bbR^{4d}} \f{\xi\cdot\zeta}{\veps^2} \, \e^{-\frac{\ri}{\veps} Q\cdot(\xi-\zeta) - \frac{\abs\xi^2}{2\veps} - \frac{\abs\zeta^2}{2\veps}} \, v(\xi+P) &
        \\&~~~~~~~~~~~~~~~~~~~~~~~~~~~~~~~~~~~~~~~~\times\overline{v}(\zeta+P)\dxi\dze \dbQ \dbP
        \\= &~ \frac{1}{2^{d}(\pi\veps)^{d/2}}\int_{\bbR^{2d}} \f{\abs\xi^2}{\veps^2} \, \e^{ - \frac{\xi^2}{\veps} } \, \abs{v(\xi+P)}^2 \dxi \dbP
        \lesssim~ \veps^{-1}\norm{v}_{L^2(\bbR^d)}^2 .
        \label{eq:R1-v-bound}
    \end{aligned}
\end{equation}
On the other hand, by the estimates \eqref{eq:bound_F} and \eqref{eq:bound_Z}, we get
\begin{equation}\label{eq:R1-psi-bound}
    \norm{ \partial_{z}\!\left(\amp\, \p_{P}^2\kindt(P) \,Z^{-1}\right) \, \e^{\ri S/ \veps} \, (\FBI \hpsieo) }_{L^2(\bbR^{2d})}^2
    \lesssim \int_{\bbR^{2d}} \left(\abs{P}^2+\delta^2\right)^{\alpha-3} \, \left|(\FBI\hpsieo)\right|^2 \dbq \dbp .
\end{equation}
By substituting \eqref{eq:R1-v-bound} and \eqref{eq:R1-psi-bound} back to \eqref{eq:R1-Holder} we obtain
\begin{align}\label{eq:R1-bound}
    \left| \langle v, \tFIO(R_1)\hpsieo \rangle \right|
     & \lesssim \veps^{-\f12} \,\norm{v}_{L^2(\bbR^{d})}
    \,\mathcal{N}_{\alpha}^\delta(t),
\end{align}
where we have defined
$\displaystyle
    \mathcal{N}_{\alpha}^\delta(t) =
    \left(\int_{\bbR^{2d}} \left(\abs{P}^2+\delta^2\right)^{\alpha-3} \, \left|(\FBI\hpsieo)\right|^2 \dbq \dbp\right)^{\f12}.
$

\subsubsection{Estimate for \texorpdfstring{$R_2$}{R2}} \label{sec:R2-bound}

Note that,
$
    \abs{\p_P^3\kindt(P)} \leq C \left(\abs{P}^2+\delta^2\right)^{\f{\alpha-3}2}.
$
Then for given $r=r(\veps)>1$,
there exist bounded functions $f_1=f_1(P,\xi)$ and $f_2=f_2(P,\xi)$ such that
\begin{equation}
    \dkindt{3}(P,\xi) =
    \begin{cases}
        f_1(P,\xi)\,g^{\delta}_\alpha(P), & \text{ if } \abs{\xi}\leq r\sqrt{\veps} ; \\
        f_2(P,\xi)\,\delta^{\alpha-3},    & \text{ if } \abs{\xi}  >  r\sqrt{\veps},  \\
    \end{cases}
\end{equation}
Here, $g^{\delta}_\alpha(P)$ is introduced to uniformly bound the singular behavior near $P=0$ by taking the minimum between the regularization parameter and the shifted momentum term, ensuring the expression remains finite:
\[
    g^{\delta}_\alpha(P) :=
    \begin{cases}
        \left(\abs{P}-r\sqrt\veps+\delta\right)^{\alpha-3}, & \text{ if } \abs{P} > r\sqrt\veps;    \\
        \delta^{\alpha-3},                                  & \text{ if } \abs{P} \leq r\sqrt\veps. \\
    \end{cases}
\]
To estimate $\norm{\tFIO\left( R_2 \right)\hpsieo}_{L^2}$, for $v\in L^2(\bbR_\xi^d)$ we compute
\begin{align}
    \langle v, \tFIO\left( R_2 \right)\hpsieo \rangle
     & =
    \frac{1}{(2\pi\veps)^{3d/2}} \int_{\bbR^{4d}} \overline{v}(\xi)\, \e^{\ri \tilde{\Phi}(t,\xi,\zeta,q,p)/\veps} \,\f{(\xi-P)^3}{6\veps^2} \,\amp\,\dkindt3\, \hpsieo(\zeta)\, \dze \dbq \dbp \dxi
    \nonumber \\& = \quad \int_{\bbR^{2d}} \dbq \dbp \;  \left(\overline{v}_1(Q,P)g_{\alpha}^{\delta}(P) + \overline{v}_2(Q,P)\delta^{\alpha-3}\right) \, \amp\, \e^{\ri S/ \veps} \, (\FBI \hpsieo)(q,p)
    \label{eq:R2-Holder}
\end{align}
where for the last equality we have defined
\begin{align*}
    {v}_1(Q,P) & := \frac{1}{2^{d}(\pi\veps)^{3d/4}}\int_{\abs\xi\leq r\sqrt\veps} \f{\xi^3}{6\veps^2} \, \e^{\frac{\ri}{\veps} Q\cdot\xi - \frac{1}{2\veps}\abs\xi^2} \, v(\xi+P) \, f_1(P,\xi)\dxi, 
    \\
    {v}_2(Q,P) & := \frac{1}{2^{d}(\pi\veps)^{3d/4}}\int_{\abs\xi   > r\sqrt\veps} \f{\xi^3}{6\veps^2} \, \e^{\frac{\ri}{\veps} Q\cdot\xi - \frac{1}{2\veps}\abs\xi^2} \, v(\xi+P) \, f_2(P,\xi)\dxi. 
\end{align*}
To use H\"older's inequalities for \eqref{eq:R1-Holder}, on one hand, we compute
\begin{subequations}\label{eq:R2-v-bound}
    \begin{align}
        \norm{{v}_1}_{L^2(\bbR^{2d})}^2
        = & ~ \frac{1}{2^{d}(\pi\veps)^{d/2}}\int_{\bbR^{2d}} \f{\abs\xi^6}{36\veps^4} \, \e^{ - \frac{\abs\xi^2}{\veps} } \, \abs{v(\xi+P)}^2\, \abs{f_1(P,\xi)}^2 \,\chi_{\abs\xi\leq r\sqrt\veps}\,\dxi \dbP
        \nonumber                                                                                                                                                                                               \\\lesssim &~ \frac{1}{(\pi\veps)^{d/2}}\int_{\bbR^{d}} \f{\abs\xi^6}{\veps^4} \, \e^{ - \frac{\abs\xi^2}{\veps} } \dxi \, \norm{v}_{L^2(\bbR^d)}^2
        \lesssim~ \veps^{-1}\norm{v}_{L^2(\bbR^d)}^2 ,
        \\
        \norm{{v}_2}_{L^2(\bbR^{2d})}^2
        = & ~ \frac{1}{2^{d}(\pi\veps)^{d/2}}\int_{\bbR^{2d}} \f{\abs\xi^6}{36\veps^4} \, \e^{ - \frac{\abs\xi^2}{\veps} } \, \abs{v(\xi+P)}^2\, \abs{f_2(P,\xi)}^2 \,\chi_{\abs\xi\geq r\sqrt\veps}\,\dxi \dbP
        \nonumber
        \\\lesssim &~ \frac{1}{(\pi\veps)^{d/2}}\int_{\abs\xi>r\sqrt\veps} \f{\abs\xi^6}{\veps^4} \, \e^{ - \frac{\abs\xi^2}{\veps} } \dxi \, \norm{v}_{L^2(\bbR^d)}^2
        \lesssim~ \veps^{\f{d}2-1}\,\emr.\,
    \end{align}
\end{subequations}
On the other hand, we estimate
\begin{subequations}\label{eq:R2-psi-bound}
    \begin{align}
        \norm{ g^{\delta}_\alpha(P)\, \amp\, \e^{\ri S/ \veps} \, \, (\FBI \hpsieo) }_{L^2(\bbR^{2d})}^2
         & \lesssim \int_{\bbR^{2d}} \abs{g_{\alpha}^{\delta}(P)}^2 \, \left|(\FBI\hpsieo)\right|^2 \dbq \dbp,
        \\
        \norm{ \delta^{\alpha-3}\, \amp\, \e^{\ri S/ \veps} \, (\FBI \hpsieo) }_{L^2(\bbR^{2d})}^2
         & \lesssim \delta^{2(\alpha-3)}\norm{\hpsieo}_{L^2(\bbR^{d})}^2.
    \end{align}
\end{subequations}
By substituting \eqref{eq:R2-v-bound} and \eqref{eq:R2-psi-bound} back to \eqref{eq:R2-Holder}, we obtain
\begin{equation}\label{eq:R2-bound}
    \abs{ \langle v, \tFIO(R_2)\hpsieo \rangle }
    \lesssim \veps^{-\f12} \left( \mathcal{G}_{\alpha}^{\delta}(t) +\veps^{\f{d}4} \emr\,\delta^{\alpha-3} \right)\norm{v}_{L^2(\bbR^{d})},
\end{equation}
where
$\displaystyle
    \mathcal{G}_{\alpha}^{\delta}(t) =
    \left(\int_{\bbR^{2d}} \abs{g_{\alpha}^{\delta}(P)}^2 \, \left|(\FBI\hpsieo)\right|^2 \dbq \dbp\right)^{\f12}.
$

\subsection*{Total remainder estimate}
Getting together \eqref{eq:RT-roll-back}, \eqref{eq:R1-bound}, and \eqref{eq:R2-bound}, we summarize these results in the following proposition for the remainder $R_\kin^\delta$:
\begin{proposition}\label{prop:RT-bound}
    For $t\in[0,\Tf]$, there exist a constant $C$ independent of $\delta$ and $\veps$, such that
    \begin{equation}\label{eq:RT-bound}
        \norm{\tFIO(R_\kin^\delta)\hpsieo}_{L^2(\bbR^d)} \leq C \, \veps^{-\f12}\,
        \left( \mathcal{N}_{\alpha}^{\delta}(t) + \mathcal{G}_{\alpha}^{\delta}(t) +
        \veps^{\f{d}4} \emr\, \delta^{\alpha-3}
        \right).
    \end{equation}
\end{proposition}

\section{Convergence of FGA to the \FSE}\label{sec:convergence}
We have gathered all the necessary estimates of the FGA quantities required for the convergence analysis. Our first theorem conducts the convergence of the FGA to the modified \FSE~\eqref{eq:Schrodinger-delta}.
\begin{theorem}\label{thm:conv-to-modfied-FSE}
    For modified \FSE~\eqref{eq:Schrodinger-delta} with $1 \leq \alpha \leq 2$, consider an initial condition $\varphi_0^{\veps}$ under Assumption \ref{asym-high-freq} and a potential function $V$ under Assumption \ref{asp:global_lipschitz}. Let $\psiedt$ and $\psie_{\FGAdt}$ be the exact solution and the FGA solution, respectively. Then for any $\delta>0$ and $t\in[0,\Tf]$,
    \begin{multline}\label{eq:psidt-psifga}
        \left\|\psiedt(t,\cdot)-\psie_{\FGAdt}(t,\cdot) \right\|_{L^2(\bbR^d)}
        \\\lesssim  \veps + \veps^{\f12}
        \left(
        \delta^{\alpha-\f52}\left(1+r\sqrt\veps/\delta\right)^{\f12} + 
        \veps^{\f{d}4} \emr\, \delta^{\alpha-3}
        \right) .
    \end{multline}
\end{theorem}
\begin{proof}
    Let $e = \psiedt-\psie_{\FGAdt}$, then by \eqref{eq:SchrFGA-delta} and \eqref{eq:SchrFGA-delta},
    \begin{equation*}
        \Big(\ri\veps\p_t-\kindt(-\ri\veps\p_x) - V(x)\Big)\,e(t,x)  = \veps^2\,\mathcal{R}^\delta\,\amp.
    \end{equation*}
    Applying Lemma 2.8 in \cite{hagedorn1998raising}, we have
    \begin{equation*}\label{eq:psidt-psifga-1}
        \begin{aligned}
            \norm{e(t,\cdot)}_{L^2(\bbR^d)}
             & \lesssim \norm{e(0,\cdot)}_{L^2(\bbR^d)}+\int_0^t \veps \norm{\left[\mathcal{R}^\delta\right](\tau,\cdot)}_{L^2(\bbR^d)} \dtau
            \\&\lesssim O(\veps^\infty)
            + \veps \int_0^t \norm{ \FIO(R^{\delta}_V)\varphi_0^{\veps} }_{L^2(\bbR^d)} \dtau
            + \veps \int_0^t \norm{ \tFIO(R_{\kin}^\delta)\hat{\varphi}_0^{\veps} }_{L^2(\bbR^d)} \dtau.
        \end{aligned}
    \end{equation*}
    By assuming $V$ is smooth enough and given $\norm{\psieo}_{L^2}=1$, we have the second term on the right side of the above estimate is bounnded by $\veps$. For the third term, note that
    \begin{equation*}\label{eq:int_Nt}
        \begin{aligned}
            \left(\int_0^t \mathcal{N}_{\alpha}^\delta(\tau)\dtau\right)^{2}
             & \lesssim \int_0^t \int_{\bbR^{2d}} \left(\abs{P(\tau,q,p)}^2+\delta^2\right)^{\alpha-3} \, \left|(\FBI\hpsieo)\right|^2 \dbq \dbp\dtau
            \\
             & = \int_{\bbR^{2d}} \int_0^t \left(\abs{P(\tau,q,p)}^2+\delta^2\right)^{\alpha-3} \dtau\, \left|(\FBI\hpsieo)\right|^2 \dbq \dbp
            \\
             & \lesssim \delta^{2\alpha-5}\int_{\bbR^{2d}}  \left|(\FBI\hpsieo)\right|^2 \dbq \dbp
            =\delta^{2\alpha-5},
        \end{aligned}
    \end{equation*}
    and
    \begin{equation*}\label{eq:int_Gt}
        \begin{aligned}
            \left(\int_0^t \mathcal{G}_{\alpha}^\delta(\tau)\dtau\right)^{2}
             & \lesssim \int_0^t \int_{\bbR^{2d}} \abs{g_{\alpha}^{\delta}(P(\tau,q,p))}^2 \, \left|(\FBI\hpsieo)\right|^2 \dbq \dbp\dtau
            \\
             & = \int_{\bbR^{2d}} \int_0^t \abs{g_{\alpha}^{\delta}(P(\tau,q,p))}^2 \dtau\, \left|(\FBI\hpsieo)\right|^2 \dbq \dbp
            \\
             & \lesssim \left(1+r\sqrt\veps/\delta\right)\delta^{2\alpha-5}\int_{\bbR^{2d}}  \left|(\FBI\hpsieo)\right|^2 \dbq \dbp
            =\left(1+r\sqrt\veps/\delta\right)\delta^{2\alpha-5}.
        \end{aligned}
    \end{equation*}
    and then by \eqref{eq:RT-bound}, we have for $1\leq\alpha\leq2$,
    \begin{equation*}\label{eq:psidt-psifga-2}
        \begin{aligned}
            \veps \int_0^t \norm{ \tFIO(R_{\kin}^\delta)\hat{\varphi}_0^{\veps} }_{L^2(\bbR^d)} \dtau
            \lesssim
            \veps^{\f12}
            \left(
            \delta^{\alpha-\f52}\left(1+r\sqrt\veps/\delta\right)^{\f12} +
            \veps^{\f{d}4} \emr\, \delta^{\alpha-3}
            \right)  .
        \end{aligned}
    \end{equation*}
    This gives the estimate \eqref{eq:psidt-psifga}.
\end{proof}

Before proceeding to the main convergence theorem, we need the following proposition conducing the difference of the original \FSE~\eqref{eq:Schrodinger} and its modified version \eqref{eq:Schrodinger-delta}.
\begin{theorem}\label{thm:Schrodinger-delta}
    For $\delta>0$, suppose $\psie(t, x)$ and $\psiedt(t, x)$ are solutions of \eqref{eq:Schrodinger} and \eqref{eq:Schrodinger-delta} respectively. Then for $t\in[0,\Tf]$, we have
    \begin{equation}\label{eq:psi-psidelta-1}
        \|\psie(t, \cdot) - \psiedt(t, \cdot)\|_{L^2(\mathbb{R}^d)} \leq C\,{\delta^{\alpha}}{\varepsilon}^{-1},
    \end{equation}
    where $C>0$ is a bounded constant that does not depend on $\varepsilon$ or $\delta$.
\end{theorem}

\begin{proof}
    We re-write the modified Schrödinger equation \eqref{eq:Schrodinger-delta} as follows:
    \begin{equation*}
        \begin{split}
            \mathrm{i}\varepsilon \partial_t \psiedt - \kin(-\ri\veps\p_x) \psiedt - V \psiedt
             & = [\kindt - \kin](-\ri\veps\p_x) \psiedt,
        \end{split}
    \end{equation*}
    For $1\leq\alpha\leq2$ it holds that
    \begin{equation*}
        \begin{aligned}
            0 \leq
            D_\delta(\xi):=\kindt(\xi) - \kin(\xi) \lesssim \left(|\xi|^2+\delta^2\right)^{\f{\alpha-1}2} \left( \sqrt{|\xi|^2+\delta^2} - |\xi| \right)
            = \f{\delta^2\left(|\xi|^2+\delta^2\right)^{\f{\alpha-1}2}}{\sqrt{|\xi|^2+\delta^2} + |\xi|}
        \end{aligned}.
    \end{equation*}
    Then it is straightforward to see that $\abs{D_\delta(\xi)}\lesssim\delta^\alpha$ uniformly. 
    By directly applying Lemma 2.8 in \cite{hagedorn1998raising}, we obtain the desired estimate \eqref{eq:psi-psidelta-1}.
\end{proof}

Now we are ready to give the main theorem of this paper.
\begin{theorem}\label{thm:conv-to-FSE}
    Consider an initial condition $\varphi_0^{\veps}$ under Assumption \ref{asym-high-freq} and a potential function $V$ under Assumption \ref{asp:global_lipschitz}. Let $\psie$ and $\psie_{\FGAdt}$ be the exact solution and the FGA solution for the \FSE~\eqref{eq:Schrodinger}, respectively.
    There exist certain choices of $\delta = \delta(\varepsilon)$, such that for $\alpha>\f{12}7$,
    \begin{align}\label{eq:psie-psifga-1}
        \left\| \psie(t, \cdot) - \psie_{\FGAdt}(t,\cdot) \right\|_{L^2(\bbR^d)}
         & \lesssim  \left(\log\veps^{-1}\right)^{\f12}\,\veps^{\f{7}{12}\alpha-1}   .
    \end{align}
\end{theorem}
\begin{proof}
    Theorems \ref{thm:conv-to-modfied-FSE} and \ref{thm:Schrodinger-delta} imply that
    \begin{align*}
         & \left\| \psie(t, \cdot) - \psie_{\FGAdt}(t,\cdot) \right\|_{L^2(\bbR^d)}
        \\ & \leq \left\| \psie(t,\cdot) - \psiedt(t,\cdot) \right\|_{L^2(\bbR^d)} + \left\| \psie_{\delta}(t,\cdot) - \psie_{\FGAdt}(t,\cdot) \right\|_{L^2(\bbR^d)}
        \\
         & \lesssim {\delta^\alpha}\,{\veps^{-1}} + \veps + \veps^{\f12} \left( \delta^{\alpha-\f52}\left(1+r\sqrt\veps/\delta\right)^\f12 + \veps^{\f{d}4} \emr\, \delta^{\alpha-3} \right)  .
    \end{align*}
    Here we have $\delta=\delta(\veps)>0$ and $r=r(\veps)>1$ to be determined.
    To guarantee convergence as $\veps\rightarrow0$, the first term requires $\delta=o(\veps^{\f1\alpha})$, and in the case $1\leq\alpha\leq2$ we can simplify the above inequality by keeping only the dominant terms
    \begin{align*}
        \left\| \psie(t, \cdot) - \psie_{\FGAdt}(t,\cdot) \right\|_{L^2(\bbR^d)}
         & \lesssim {\delta^\alpha}\,{\veps^{-1}} + r^{\f12}\,\veps^{\f34} \delta^{{\alpha-3}}  + \veps^{1/2}\,\emr\,\delta^{\alpha-3}.
    \end{align*}
    Choose $\delta=\veps^{\beta}$ and $r=r(\veps)>1$ such that the three terms on the right side of the above inequality balanced.
    To balance $\veps^{\f34+\beta(\alpha-3)}r^{\f12}$ and $\veps^{\f12+\beta(\alpha-3)}\,\emr$, we need $r\sim\left(\log\veps^{-1}\right)^{\f12}$.
    Then to balance $\veps^{\alpha\beta-1}$ and $\veps^{\f34+\beta(\alpha-3)}\,\left(\log\veps^{-1}\right)^{\f12}$, we need $\beta=\f{7}{12}$. Then it is easy to check for $\alpha>\f{12}7$ we have
    \begin{align*}
        \left\| \psie(t, \cdot) - \psie_{\FGAdt}(t,\cdot) \right\|_{L^2(\bbR^d)}
         & \lesssim  \left(\log\veps^{-1}\right)^{\f12}\,\veps^{\f{7}{12}\alpha-1}   .
    \end{align*}
    This completes the proof of the convergence result \eqref{eq:psie-psifga-1}.
\end{proof}

In the case of a linear potential, the above convergnece results can be further improved as follows.
\begin{theorem}\label{thm:conv-to-FSE-linear}
    For the \FSE~\eqref{eq:Schrodinger} with $1 < \alpha \leq 2$, consider an initial condition $\varphi_0^{\veps}$ under Assumption \ref{asym-high-freq} and a potential function $V(x) = c_0 + c_1\cdot x$ for some constants $c_0\in\bbR$ and $0\neq c_1\in\bbR^d$. Let $\psie$ and $\psie_{\FGAdt}$ be the exact solution and the FGA solution, respectively.
    There exist certain choices of $\delta = \delta(\varepsilon)$ such that
    \begin{align}\tag{\ref{eq:psie-psifga-1}'}\label{eq:psie-psifga-linear}
        \left\| \psie(t, \cdot) - \psie_{\FGAdt}(t,\cdot) \right\|_{L^2(\bbR^d)}
         & \lesssim  \left(\log\veps^{-1}\right)^{\f34}\,\veps^{\f{\alpha-1}{2}}   .
    \end{align}
\end{theorem}
\begin{proof}
    We focus our proof on the case of Gaussian initial conditions, i.e., $\varphi_0^{\veps}(x) = g^\veps(x):=(\pi\veps)^{-d/4}\,e^{-\abs{x-q_0}^2/(2\veps)}\,e^{\ri p_0\cdot(x-q_0)/\veps}$ for some $q_0,p_0\in\bbR^d$. In this case, we have an explicit formula for the FBI transform of $g^{\veps}$:
    \begin{equation*}
        \tilde{g}(q,p):=(\FBI g^{\veps})(q,p) = (2\pi\veps)^{-\f{d}2} \, \e^{\f{\ri}{2\veps} (p+p_0)\cdot(q-q_0)} \, \e^{-\f{\abs{q-q_0}^2+\abs{p-p_0}^2}{4\veps}} .
    \end{equation*}
    Then for a general initial condition $\varphi_0^{\veps}$ under Assumption \ref{asym-high-freq}, we can express $\varphi_0^{\veps}$ as a superposition of Gaussian wave packets via the (pseudo-) inverse FBI transformation \eqref{eq:FBI_inv}. Thus, by linearity of both the \FSE~\eqref{eq:Schrodinger-delta} and the FGA approximation, the convergence results still holds.

    In the case of a linear potential $V(x) = c_0 + c_1\cdot x$ for some constants $c_0\in\bbR$ and $0\neq c_1\in\bbR^d$, noting that $P(t,q,p) = p+c_1t$, we have more refined estimates for $\mathcal{N}_\alpha^\delta$ and $\mathcal{G}_\alpha^\delta$ than those in the proof of Theorem \ref{thm:conv-to-modfied-FSE}. For simplicity, we only present the $d=1$ case. In fact, for $1\leq\alpha\leq2$, we have
    \begin{equation*}
        \begin{aligned}
            \left(\mathcal{N}_{\alpha}^\delta(t)\right)^{2}
             & = \int_{\bbR^{2}} \left(\abs{p-c_1t}^2+\delta^2\right)^{\alpha-3} \, \left|(\FBI\hpsieo)\right|^2 \dbq \dbp
            \\ & \leq \int_{\bbR^{2}} \left(\abs{p-c_1t}^2+\delta^2\right)^{\alpha-3} \,  {(2\pi\veps)^{-{1}}}{\e^{-\f{\abs{q-q_0}^2+\abs{p-p_0}^2}{2\veps}}}  \dbq \dbp
            \\ & = {(2\pi\veps)^{-\f{1}2}}\int_{\bbR} \left(\abs{p+p_0-c_1t}^2+\delta^2\right)^{\alpha-3} \,  {\e^{-\f{\abs{p}^2}{2\veps}}}    \dbp\,.
        \end{aligned}
    \end{equation*}
    To estimate this integral, for $\alpha\in(1,2)$, we decompose the integral by $\abs{\xi}=r\veps$ and compute
    \begin{equation*}
        \begin{aligned}
            I(t)
             & :=\int_{\abs{p}\geq r\sqrt\veps}~+~\int_{\abs{p}< r\sqrt\veps} \left(\abs{p+p_0-c_1t}^2+\delta^2\right)^{\alpha-3}  \e^{-\f{\abs{p}^2}{2\veps}} \dbp
            \\&\leq C \, {\delta^{2\alpha-6}} {\e^{-{r^2}/2}} + \int_{\abs{p}<r\sqrt\veps} {\left(\abs{p+p_0-c_1t}+\delta\right)^{2\alpha-6}}  \dbp
            \\&\lesssim {\delta^{2\alpha-6}} {\e^{-{r^2}/2}} +
            \left\{
            \begin{aligned}
                 & \,\delta^{2\alpha-5}                                       & , & \quad \text{if } \abs{p_0-c_1t}\leq r\sqrt\veps; \\
                 & \left(\abs{p_0-c_1t}-r\sqrt\veps+\delta\right)^{2\alpha-5} & , & \quad \text{if } \abs{p_0-c_1t}> r\sqrt\veps.
            \end{aligned}
            \right.
        \end{aligned}
    \end{equation*}
    So we have for $\delta<O(\sqrt\veps)$,
    \begin{equation*}
        \begin{aligned}
            \int_0^t \mathcal{N}_{\alpha}^\delta(\tau) \dtau
             & = (2\pi\veps)^{-\f{1}4} \int_0^t  I(\tau)^{\f12} \dtau
             & \lesssim \veps^{-\f{1}4}
            \left(
            r\,\sqrt\veps\,\delta^{\alpha-\f52}
            \right) + \veps^{-\f14}\delta^{\alpha-3}\,\e^{-{r^2/2}}.
        \end{aligned}
    \end{equation*}
    And similarly,
    for $\mathcal{G}_\alpha^\delta(t)$:
    \begin{equation*}
        \begin{aligned}
            \left(\mathcal{G}_{\alpha}^\delta(t)\right)^{2}
             & = \int_{\bbR^{2}} \left(g_{\alpha}^{\delta}(P(t))\right)^{2} \, \left|(\FBI\hpsieo)\right|^2 \dbq \dbp
            \\ & \leq {(2\pi\veps)^{-\f{1}2}}\int_{\bbR} \left(g_{\alpha}^{\delta}({p+p_0-c_1t})\right)^{2} \,  {\e^{-\f{\abs{p}^2}{2\veps}}}    \dbp\,.
        \end{aligned}
    \end{equation*}
    We decompose the integral again by $\abs{p}=r\sqrt\veps$ and estimate
    \begin{equation*}
        \begin{aligned}
             & \int_{\abs{p}< r\sqrt\veps} \left(g_{\alpha}^{\delta}({p+p_0-c_1t})\right)^{2} \,  {\e^{-\f{\abs{p}^2}{2\veps}}} \dbp
            \\&\lesssim
            \left\{
            \begin{aligned}
                 & \,r\,\sqrt\veps\,\delta^{2\alpha-6}                        & , & \quad \text{if } \abs{p_0-c_1t}\leq r\sqrt\veps; \\
                 & \left(\abs{p_0-c_1t}-r\sqrt\veps+\delta\right)^{2\alpha-5} & , & \quad \text{if } \abs{p_0-c_1t}> r\sqrt\veps.
            \end{aligned}
            \right.
        \end{aligned}
    \end{equation*}
    Thus,
    we can estimate $\mathcal{G}_\alpha^\delta(t)$ as
    \begin{equation*}
        \begin{aligned}
            \int_0^t \mathcal{G}_{\alpha}^\delta(\tau) \dtau
             & \lesssim \veps^{-\f14}
            \left( r\,\sqrt\veps \, r^{\f12}\,\veps^{\f{1}4}\,\delta^{\alpha-3}
            + \delta^{\alpha-3}\,\e^{-{r^2/2}}\right).
        \end{aligned}
    \end{equation*}
    Thus in the linear potential case, Theorem \ref{thm:conv-to-modfied-FSE} can be refined to
    \begin{equation}\tag{\ref{eq:psidt-psifga}'}\label{eq:psidt-psifga-linear}
        \left\|\psiedt(t,\cdot)-\psie_{\FGAdt}(t,\cdot) \right\|_{L^2(\bbR^d)}
        \lesssim  \veps +
        r^{\f32}\veps\,\delta^{\alpha-3} + 
        \veps^{\f12+\f{d}4} \emr\, \delta^{\alpha-3} .
    \end{equation}

    Furthermore, the estimate in Theorem \ref{thm:Schrodinger-delta} can also be refined in the linear potential case. If $V(x) = c_0 + c_1\cdot x$ for some constants $c_0\in\bbR$ and $0\neq c_1\in\bbR^d$, then $\psiedt$ can be solved explicitly as
    \begin{equation*}
        \hat\psiedt(t,\xi) = \hpsieo(\xi+c_1 t)\exp\left(-\frac{\ri}{\veps}\int_0^t\kindt(\xi+c_1(t-s))\,\rd s - \f\ri\veps c_0t\right).
    \end{equation*}
    Therefore,
    \begin{equation*}
        \begin{aligned}
            \| D_\delta(\cdot) \hat\psiedt(t,\cdot) \|_{L^2(\bbR^d)}^2
             & = \int  D_\delta^2(\xi) \abs{\hpsieo(\xi+c_1t)}^2 \dxi
            \\& = \int  D_\delta^2(\xi+p_0-c_1t) (\pi\veps)^{-d/2}\,e^{-\f{\abs{\xi}^2}{2\veps}} \dxi .
        \end{aligned}
    \end{equation*}
    Then a similar argument as previously fo $\mathcal{N}$ gives
    \begin{equation*}
        \int_0^t\| D_\delta(\cdot) \hat\psiedt(\tau,\cdot) \|_{L^2(\bbR^d)}\dtau
        \lesssim
        r\,\veps^{\f14}\,\max\left\{\delta^2,\delta^{\alpha+\f12}\right\} + \veps^{-\f14}\delta^{\alpha-2}\,\e^{-{r^2/2}}.
    \end{equation*}
    Thus in the linear potential case, Theorem \ref{thm:Schrodinger-delta} can be refined to
    \begin{equation}\tag{\ref{eq:psi-psidelta-1}'}\label{eq:psi-psidelta-linear}
        \|\psie(t, \cdot) - \psiedt(t, \cdot)\|_{L^2(\mathbb{R}^d)} \lesssim r\,{\varepsilon}^{-\f34} \max\left\{\delta^2,\delta^{\alpha+\f12}\right\} + \veps^{-\f54}\delta^{\alpha-2}\,\e^{-{r^2/2}}.
    \end{equation}

    Then by combining \eqref{eq:psidt-psifga-linear} and \eqref{eq:psi-psidelta-linear}, and by choosing $r\sim\left(\log\veps^{-1}\right)^{\f12}$ and $\delta=\veps^{\f12}$ , we have for $1<\alpha\leq 2$
    \begin{equation*}
        \begin{aligned}
            \left\| \psie(t, \cdot) - \psie_{\FGAdt}(t,\cdot) \right\|_{L^2(\bbR^d)}
             & \lesssim r^{\f32}\,\veps^{\f{\alpha-1}2} .
        \end{aligned}
    \end{equation*}
    which completes the proof.
\end{proof}


We conclude this section with some further remarks.

\begin{remark}
    The modified Schr\"odinger equation is introduced to remove the singularity of the fractional Laplacian, serving as an intermediate step in proving the convergence of the FGA.
    Additionally, the modified model \eqref{eq:Schrodinger-delta} is of independent interest in the study of semi-relativistic boson stars \cite{frohlich2007boson,lenzmann2007well}.  It is valuable to consider the convergence of the FGA solution $\psie_\FGAdt$ to the solution $\psiedt$ of \eqref{eq:Schrodinger-delta} for a given fixed $\delta>0$. In this context, the result in Theorem \ref{thm:conv-to-modfied-FSE} is far from optimal. Actually, for a fixed $\delta$, $\kindt$ can be regarded as sufficiently smooth, which ensures that $R_\kin^\delta$ in the form of \eqref{eq:RT} is $O(1)$. The analysis of \eqref{eq:SchrFGA-delta} becomes significantly simpler compared to the analysis performed to prove Poroposition \ref{prop:RT-bound} and Theorem \ref{thm:conv-to-modfied-FSE}. Consequently, $O(\veps)$-convergence of the FGA can be demonstrated. 

    However, when $\delta=\delta(\varepsilon)\to0$ as $\varepsilon\to0$, the estimate above, although suboptimal, is crucial: it permits passing from the regularized model back to the original problem and thus proving convergence of $\psi^\varepsilon_{\delta,\mathrm{FGA}}$ to $\psi^\varepsilon$ in Theorem \ref{thm:conv-to-FSE}. The deterioration of the proven rate compared with the optimal $O(\varepsilon)$ known for the standard Schr\"odinger equation stems from technical steps in our analysis. In particular, repeated integration by parts and the need to control high-order derivatives in $(q,p)$ to extract explicit powers of $\varepsilon$ (see \eqref{eq:RT}) lead to losses, while Proposition \ref{prop:bound_F} only provides uniform bounds for zeroth- and first-order derivatives. These extra differentiations do not appear in the final FGA ansatz and therefore seem to be an artefact of the estimation strategy rather than an intrinsic limitation of the method. In particular, for a linear potential one can obtain improved bounds and prove convergence for $1<\alpha\leq2$. Numerical results in Section \ref{sec:num} show near-linear decay in $\varepsilon$, suggesting the theoretical rate can likely be sharpened.

    Motivated by these observations, we conjecture that the FGA for the \FSE\ attains the optimal $O(\varepsilon)$ error for $1<\alpha\leq2$, as in the classical Schr\"odinger case. Achieving this sharper result would require strengthening several technical ingredients in our analysis:
    (i) perform a more refined analysis of $R_\kin^\delta$, using sharper decompositions in $(\xi-P)$, localized cutoffs, and optimized H\"older-type estimates to replace the current coarse bounds in Proposition \ref{prop:RT-bound};
    and (ii) generalize the arguments in Theorem \ref{thm:conv-to-FSE-linear} for linear potentials to general potentials by leveraging the localized nature of the solution ansatz in phase-space.
    Obtaining either of the improvements above would directly strengthen the error bounds and could restore the conjectured $O(\varepsilon)$ convergence rate. We leave these refinements for future work.

    Furthermore, our analysis can be extended to any Schr\"odinger-type equations defined by pseudo-differential operators, provided that the symbol is either smooth or has singularities that behave similarly to the fractional Laplacian.
\end{remark}

\section{Numerical tests}\label{sec:num}

In this section, we provide numerical examples in both one and two dimensions to demonstrate the accuracy of FGA for the \FSE~and verify its convergence rate with respect to $\veps$.

In our numerical experiments, we use FGA to compute solutions for various values of $\veps$ and compare them with the reference solutions obtained using the time-splitting spectral method.
The initial condition is selected in the classical WKB form, \ie,
\begin{equation*}
    \psi^{\veps}(x, t=0)=\varphi_0^{\veps}(x) = \sqrt{n_0(x)} \e^{\ri S_0(x)/\veps},
\end{equation*}
where $n_0(x)$ and $S_0(x)$ are independent of $\veps$ and real valued. 


The modified FGA approximates the solution to the \FSE~\eqref{eq:FSE} by \eqref{eq:fga_standard},
where the FGA variables $Q$, $P$, $S$, and $a$ satisfy \eqref{eq:hamiltonian-flow}, \eqref{eq:action}, and \eqref{eq:amplitude}, with the modified Hamiltonian
\begin{equation*}
    H=\frac{1}{\alpha} (|P|^2+\delta^2 )^{\alpha /2} + V(Q).
\end{equation*}
The FGA algorithm first decomposes the initial wave into Gaussian functions in the phase space. Subsequently, it propagates the center of each function along the characteristic lines. Finally, the solution is reconstructed through integration over the phase space. This procedure involves using discrete meshes for $x$, $q$, $p$, and $y$ with step size $\Delta x$, $\Delta q$, $\Delta p$, and $\Delta y$ under appropriate strategies.

Lastly, to evolve the ODEs of FGA when the trajectory encounters the singularity at $|P|=0$, we introduce a parameter $\delta$ in the modified FGA formulation. Numerically, we need to determine the appropriate value for $\delta$. According to our theoretical convergence proof, it should take the form $\delta=\varepsilon^k$, where $k>0$ influences the convergence behavior. Our convergence proof gives $k = \frac{7}{12}$. However, this choice may not be optimal. In the experiments presented below, we verify the accuracy and convergence rate of FGA with different $\delta$ values. Notably, the $\delta=\veps$ (\ie, $k=1$) provides the desired numerical accuracy. 

\subsection{One dimension}

\begin{example} \label{example:one_dim}
    The initial condition is given by
    \begin{equation}
        \varphi_0^\veps(x) = \sqrt{\frac{64}{\pi}} \exp\left(- 64(x- 1)^2 \right) \exp\left( \frac{\ri x}{\veps} \right) , \quad x\in \bbR.
    \end{equation}
    We solve the equation on the $x$-interval $[0, 2]$ with final time $T=0.25$, under the potential
    \begin{equation}
        V(x) = 1 + \cos(\pi x).
    \end{equation}
\end{example}

In this example, we aim to validate FGA solutions' accuracy and convergence behavior. We set the mesh sizes as follows: $\Delta x=\Delta y=\veps$, $\Delta q=O(\sqrt{\veps})$,and  $\Delta p =O(\sqrt{\veps})$. The time step for solving the ODEs using the $4$-th order Runge-Kutta method is $\Delta t=10^{-2}$. The reference solution is computed by the time-splitting spectral method with a mesh size of $\Delta x=\veps$ and a time step of $\Delta t = \veps^2$.
We note that, in this example, we have chosen small mesh sizes for both $p$ and $q$, with a sufficiently large number of grid points. This ensures that the primary source of error in the FGA solution is the asymptotic expansion rather than the initial decomposition, numerical integration of ODEs, or other factors. However, such a fine mesh selection is not strictly necessary for achieving accurate results with FGA. Furthermore, we have confirmed the presence of trajectories $P(t,q,p)$ with distinct initial points $(q,p)$ that pass through zero during the evolution within the specified final time, based on the intermediate value theorem for continuous functions. The existence of such trajectories ensures the meaningfulness of this numerical illustration.

We numerically investigate the convergence rate with respect to $\veps$.
Specifically, we plot the decay curves of $\log_2(\|\psie_{\FGA} - \psie_{\text{ref}}\|_2)$ versus $-\log_2(\veps)$ for each $\alpha$ in the range of $1.1, 1.3, 1.5, 1.7, 1.9$ in Figure \ref{fig:err_rate_1d}. The slopes of these curves are determined using the least squares method. In the first subfigure of Figure \ref{fig:err_rate_1d}, we set $\delta=\veps$, while in the second subfigure, we set $\delta=\veps^{7/12}$, as suggested by the proof of the convergence theorem \ref{thm:conv-to-FSE}. Remarkably, in the $\delta=\veps$ scenario, the FGA displays a linear decay rate for the $L^2$ errors, which is consistent with the convergence rate observed for the standard Schr\"odinger equation. In contrast, in the $\delta=\veps^{7/12}$ scenario, the logarithm of $L^2$ error decays at a slower rate, indicating that the FGA solution converges, but the inequality bound in our proof is not very tight, leaving room for further improvement. Despite this suboptimal choice for $\delta$, the FGA exhibits desirable numerical properties and convergence rates for practical problems.

\begin{figure}
    \centering
    \subfigure[$\delta=\veps$]{
        \includegraphics[width=0.45\textwidth]{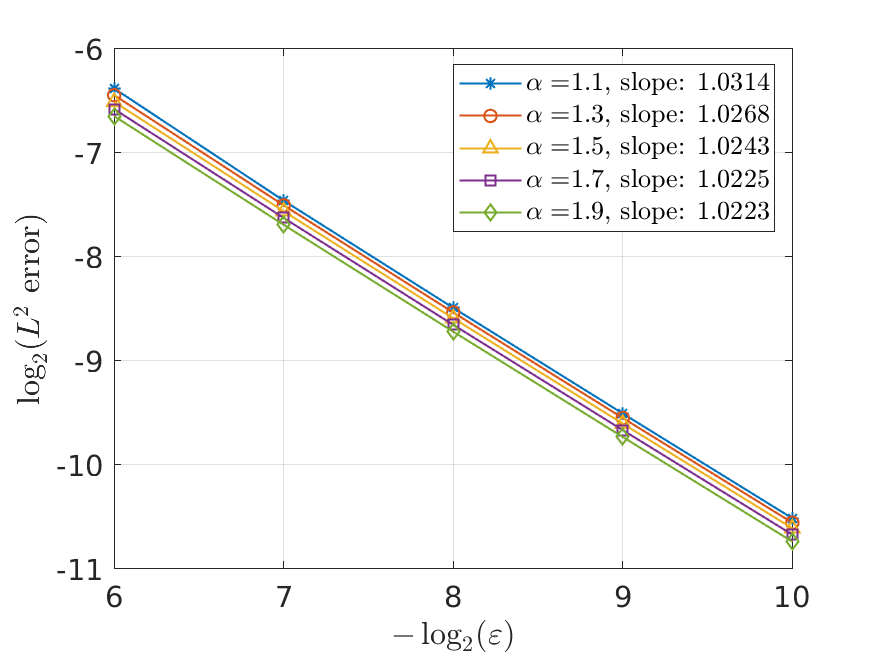}
    }
    \subfigure[$\delta=\veps^{7/12}$]{
        \includegraphics[width=0.45\textwidth]{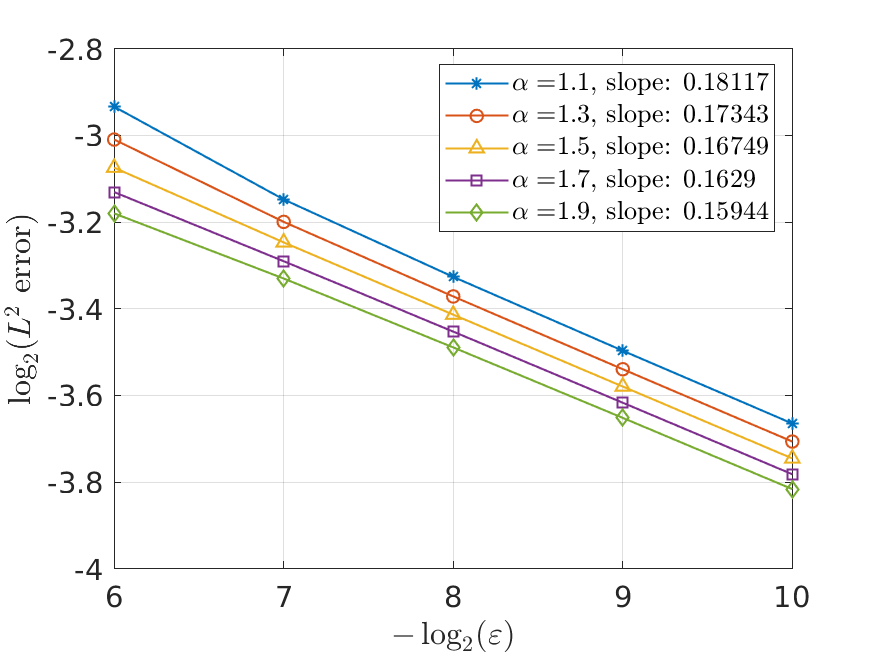}
    }
    \caption{$L^2$ error decay behavior of FGA solution with different $\delta$ values in Example \ref{example:one_dim}.}
    \label{fig:err_rate_1d}
\end{figure}

\subsection{Two dimension}

\begin{example} \label{example:two_dim}
    The initial condition is given by
    \begin{equation}
        \varphi_0^\veps(x_1, x_2) = \frac{64}{\pi} \exp \left(-64\left((x_1 -1)^2 + (x_2 -1)^2)\right)\right) \exp \left( \frac{\ri(x_2 - 1)}{\veps}\right), x_1, x_2 \in \bbR.
    \end{equation}
    We solve the equation on domain $(x_1, x_2)\in [0, 2]^2$ with final time $T=0.25$, under potential
    \begin{equation}
        V(x_1, x_2) = \frac{1}{2} \left( (x_1 - 1)^2 + (x_2 - 1)^2 \right).
    \end{equation}
\end{example}

In this example, we extend the discussion from Example \ref{example:one_dim} to demonstrate the numerical behavior of FGA in the two-dimensional scenario. We set the mesh sizes as follows: $\Delta x_1=\Delta x_2=\Delta y_1=\Delta y_2 =\veps$, $\Delta q_1=\Delta q_2=O(\sqrt{\veps})$, and $\Delta p_1=\Delta p_2 =O(\sqrt{\veps})$. The time step for solving the ODEs with $4$-th order Runge-Kutta method is $\Delta t=10^{-2}$. The reference solution is computed by the time-splitting spectral method with mesh sizes $\Delta x_1=\Delta x_2=\veps$ and a time step of $\Delta t = \veps^2$.

We plot the decay curves of $\log_2(\|\psie_{\FGA}-\psie_{\text{ref}}\|_2)$ versus $-\log_2(\veps)$ for each $\alpha=1.1$, $1.3$, $1.5$, $1.7$, $1.9$ in Figure \ref{fig:err_rate_2d}. The first subfigure uses $\delta=\veps$ and the second subfigure uses $\delta=\veps^{7/12}$, which is derived from our proof for the convergence theorem \ref{thm:conv-to-FSE}. Similar to the one-dimensional example case, the $\delta=\veps$ scenario displays a linear convergence rate, while the $\delta=\veps^{7/12}$ shows a slower decay rate.

\begin{figure}
    \centering
    \subfigure[$\delta=\veps$]{
        \includegraphics[width=0.45\textwidth]{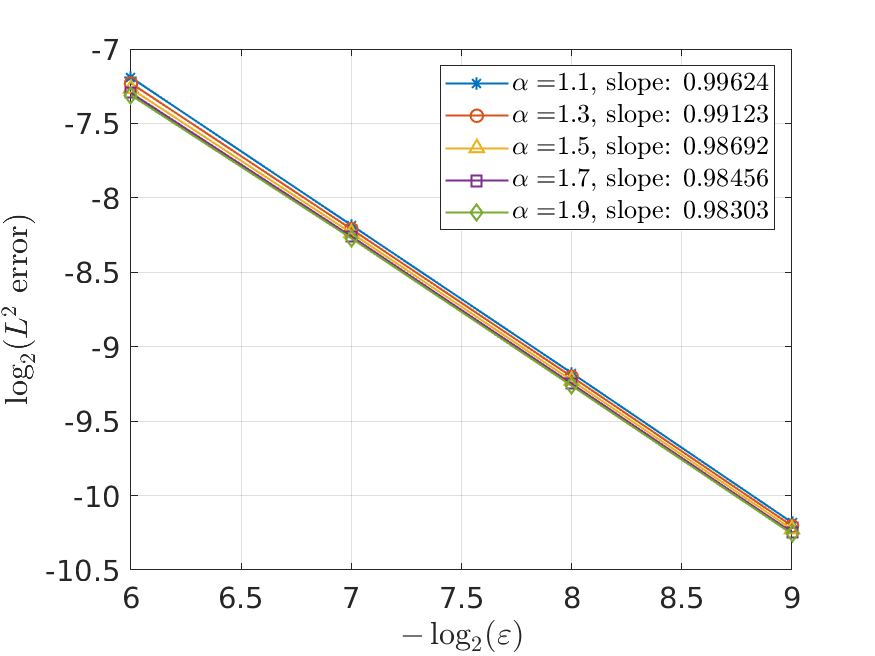}
    }
    \subfigure[$\delta=\veps^{7/12}$]{
        \includegraphics[width=0.45\textwidth]{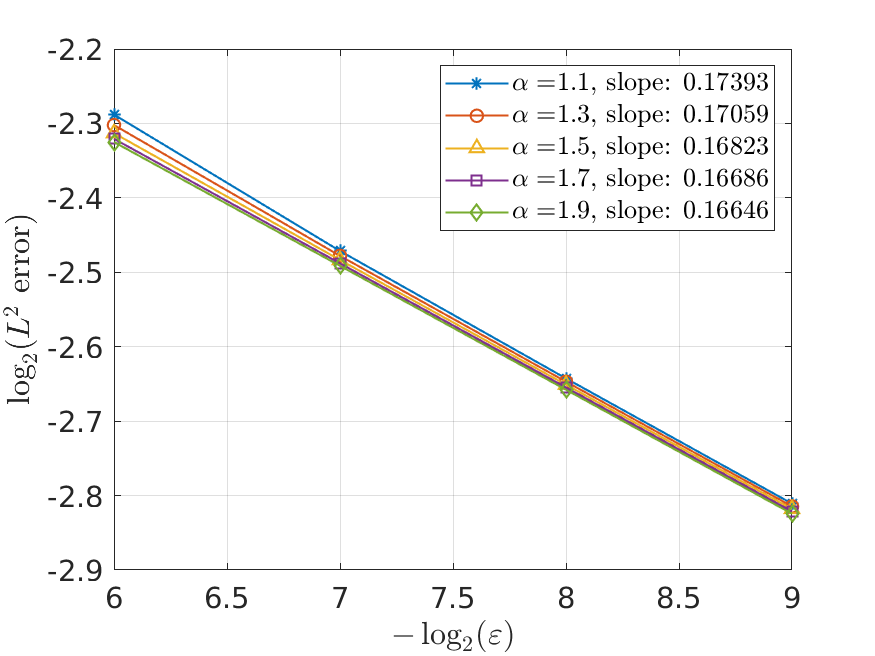}
    }
    \caption{$L^2$ error decay behavior of FGA solution in Example \ref{example:two_dim}.}
    \label{fig:err_rate_2d}
\end{figure}

\section{Conclusion} \label{sec:conclusion_discussion}

In this work, we propose the frozen Gaussian approximation (FGA) for computing the \FSE~in the semi-classical regime. This approach is based on asymptotic analysis and constructs the solution using fixed-width Gaussian functions in the phase space. We develop a modified FGA formulation with a regularization parameter $\delta$ to address the presence of singularities in higher derivatives of the associated Hamiltonian flow. Our primary focus is on deriving the formulations and establishing the rigorous convergence result for the FGA. Additionally, we verify the accuracy of the FGA solutions through several numerical examples and demonstrate the method's linear convergence rate with respect to $\varepsilon$.

\section*{Acknowledgements}
L.C. was partially supported by
the National Key R\&D Program of China No.2021YFA1003001, 
and the NSFC Projects No. 12271537 and 11901601. X.Y. was partially supported by the NSF grant DMS-2109116.

\appendix

\section{Proof of Proposition \ref{prop:bound_F}} \label{app:bounds-F}

\subsection{Parities}
Consider the trajectory of the Hamiltonian flow \eqref{eq:hamiltonian-flow} such that $P(t_0)=0$ and $Q(t_0)=Q_0$ for some $0<t_0<\Tf$.
Let $Q^{(0)}(t)\equiv Q_0$ and $P^{(0)}(t)\equiv 0$, then the Picard's iteration
\begin{align*}
    Q^{(k+1)}(t) = Q_0 + \int_{0}^t\p_P\kindt\left(P^{(k)}(s)\right)\rd s,
    \quad
    P^{(k+1)}(t) = 0 - \int_{0}^t\p_QV\left(Q^{(k)}(s)\right)\rd s,
\end{align*}
produce a sequence of $(Q^{(k)},P^{(k)})$ where $Q^{(k)}$ is even and $P^{(k)}$ is odd w.r.t $t=t_0$, for all $k=0,1,2,...$.   Thus $Q(t)$ and $P(t)$ are respectively even and odd functions in $t$ w.r.t $t=t_0$. In the following discussions we will consider the corresponding auxiliary quantities associated with this given trajectory  $(Q(t), P(t))$, and by default, an even or odd function $f$ means the function is even or odd in time variable $t$ w.r.t. $t=t_0$, \ie, $f(t_0+s)=f(t_0-s)$ for all $s\in\bbR$.

\subsection{The estimate for \texorpdfstring{$F$}{F}}

By differentiating \eqref{eq:hamiltonian-flow}, one get the linear ODE for $F$:
\begin{align} \label{eq:dFdt}
    \ddt{}F= J \nabla^2 H F
\end{align}
with initial condition $F(t=0) = \mathrm{Id}_{2d}$.

According to Prop. \ref{prop:bound_PQ}, there exist constant $c_1$ and $c_2$ such that $c_1(t-t_0)\leq \abs{P(t)}\leq c_2(t-t_0)$ for $t\in[0,\Tf]$. Thus,
\begin{equation}\label{eq:P=t}
    C_1 \left(\delta+\abs{t-t_0}\right)
    \leq\left(\abs{P}^2+\delta^2\right)^{1/2}\leq
    C_2\left(\delta+\abs{t-t_0}\right).
\end{equation}
Therefore,  we have
\begin{equation}
    \ddt{}|F(t, q, p)| \leq \abs{\nabla^2H}
    |F(t, q, p)|  \leq C { \left(\delta + \abs{t-t_0}\right)^{{\alpha-2}} } |F(t, q, p)|.
\end{equation}
Since $1<\alpha\leq 2$, integrating both sides of the above inequality gives
\begin{align}
    |F(t, q, p)|\leq \abs{F(0,q,p)}\e^{C|t-t_0|^{\alpha-1}}\leq C.
\end{align}

Let $X$ be the fundamental solution matrix of \eqref{eq:dFdt} such that $X(t_0) = \mathrm{Id}_{2d}$. The same estimate can be applied to $X$ and we find that both $|X|$ and $|F|$ are bounded by $C$ independent of $\delta$. Thus there exist a constant (in time) matrix $Y$ that is bounded independently of $\delta$, such that $F=XY$.
Furthermore, suppose $X=\begin{pmatrix} X_{qq}& X_{qp} \\ X_{pq} & X_{pp}\end{pmatrix}$ is partitioned in a conformal way as $F$ in \eqref{eq:F}, then it is easy to verify that: $X_{qq}$ and $X_{pp}$ are even, while $X_{qp}$ and $X_{pq}$ are odd, noting that $\nabla^2H$ is even w.r.t. $t=t_0$.
These properties will help us to get a fine estimate for $\p F$ in the next subsection.


\subsection{The estimate for \texorpdfstring{$\p F$}{F1}}

Differentiating \eqref{eq:dFdt} gives
\begin{equation} \label{eq:dF1dt}
    \ddt{} \pa F= J \nabla^2 H \,\pa F+J\,\pa\!\nabla^2 H \,F,
\end{equation}
where $\pa$ stands for $\p_{q_j}$ or $\p_{p_j}$, $j= 1,2,...,d$.
Using the fundamental solution matrix $X$, the solution of this system can be written as
\begin{align*}
    \pa F(t) & = X\int_{0}^tX^{-1}(s)J\,\pa\!\nabla^2H(s)\,F(s)\rd s
    = X\int_{0}^t J^T X^{T}(s)J\,J\,\pa\!\nabla^2H(s)\,F(s)\rd s  \nonumber                                                           \\
             & = X\,J^T\int_{0}^t X^{T}(s)\,\pa\!\nabla^2H(s)\,F(s)\rd s = X\,J^T\int_{0}^t X^{T}(s)\,\pa\!\nabla^2H(s)\,X(s)\rd s\,Y
    \label{eq:dF1dt_solution},
\end{align*}
where for the second equality we have used the symplecticity of $X$.

Let $G(s)=X^{T}(s)\,\pa\!\nabla^2H(s)\,X(s)$, then
\begin{align}
    G = X^T\,\begin{pmatrix}
                 \pa Q \, \p_Q^3V & \\ & 0
             \end{pmatrix}\, X
    +X^T\,\begin{pmatrix}
              0 & \\ & \pa P \, \p_P^3\kindt
          \end{pmatrix}\, X =:G_{\text{bdd}} + G_{\delta}.
\end{align}

The previous subsection has shown that $X$ and $F$ are bounded independent of $\delta$, then  it implies
\begin{align}\label{eq:G_bdd_bound}
    \int_0^t \abs{ G_{\text{bdd}} } \,\rd s =\int_0^t  \abs{ X^T\,\begin{pmatrix}  \pa Q \, \p_Q^3V & \\ & 0 \end{pmatrix}\, X }\, \rd s \leq C
\end{align}
under the assumption that $V$ is smooth.

On the other hand, since $\pa P(t) = X_{pp}^T(t)\,Y_{pa} +  X_{pq}^T(t)\,Y_{qa}$ with $X_{pp}$ and $X_{pq}$ being even and odd respectively, we can split $G_{\delta}$ into its even and odd parts by $G_{\delta}= G_{\delta,\text{e}} + G_{\delta,\text{o}}$ where
\begin{align*}
    G_{\delta,\text{e}} & =
    \begin{pmatrix}
        X_{pq}^T\,Y_{qa}^TX_{pq}^T\, \p_P^3\kindt\, X_{pq} & X_{pq}^T\,Y_{pa}^TX_{pp}^T\, \p_P^3\kindt\, X_{pp} \vspace{3pt} \\
        X_{pp}^T\,Y_{pa}^TX_{pp}^T\, \p_P^3\kindt\, X_{pq} & X_{pp}^T\,Y_{qa}^TX_{pq}^T\, \p_P^3\kindt\, X_{pp}
    \end{pmatrix}, \quad\text{and}
    \\
    G_{\delta,\text{o}} & =
    \begin{pmatrix}
        X_{pq}^T\,Y_{pa}^TX_{pp}^T\, \p_P^3\kindt\, X_{pq} & X_{pq}^T\,Y_{qa}^TX_{pq}^T\, \p_P^3\kindt\, X_{pp} \vspace{3pt} \\
        X_{pp}^T\,Y_{qa}^TX_{pq}^T\, \p_P^3\kindt\, X_{pq} & X_{pp}^T\,Y_{pa}^TX_{pp}^T\, \p_P^3\kindt\, X_{pp}
    \end{pmatrix}.
\end{align*}
For the even part $G_{\delta,\text{e}}$, we observe that each of its blocks has at least one $X_{pq}$ factor. We know that $X_{pq}$ is odd and thus $X_{pq}(t_0)=0$. In addition, by \eqref{eq:dFdt},
\begin{align*}
    \abs{\ddt{X_{pq}}} \leq \abs{\p_Q^2 V}\,\abs{X_{qq}} \leq C, \quad\text{for all }t\in[0,\Tf],
\end{align*}
where for the last inequality we have used the fact that $X_{qq}$ is bounded and $V$ is assumed to be smooth independent of $\delta$.
Thus
\begin{align*}
    \abs{X_{pq}(t)}\leq C\abs{t-t_0}.
\end{align*}
Then
\begin{align}
    \int_0^t\abs{G_{\delta,\text{e}}(s)}\rd s
     & \leq C\int_0^t\abs{(s-t_0)\p_P^3\kindt(s)}\rd s
    \leq C\int_0^t\abs{(s-t_0)} \left(\delta+\abs{s-t_0}\right)^{\alpha-3}\rd s
    \nonumber                                                        \\
     & \leq C\int_0^t\left(\delta+\abs{s-t_0}\right)^{\alpha-2}\rd s
    \leq C, \quad \text{ for } 1<\alpha<2.
    \label{eq:G_even_bound}
\end{align}
For the odd part $G_{\delta,\text{o}}$, the $(p,p)$-block does not include a $X_{pq}$ factor, which makes the above estimate invalid. A practical upper bound for $G_{\delta,\text{o}}$ is
$$
    \abs{ G_{\delta,\text{o}} } \leq C \abs{\p_P^3\kindt} \leq C\left(\delta+\abs{s-t_0}\right)^{\alpha-3}.
$$
For $t\leq t_0$,
\begin{align*}
    \abs{ \int_0^t G_{\delta,\text{o}}(s)\,\rd s }
     & \leq C \int_0^t \left(\delta+\abs{s-t_0}\right)^{\alpha-3}\,\rd s
    = C \int_0^t \left(\delta+t_0-s\right)^{\alpha-3}\,\rd s
    \\&\leq C\left(\left(\delta+t_0-t\right)^{\alpha-2}-\left(\delta+t_0\right)^{\alpha-2}\right).
\end{align*}
For $t>t_0$, using the fact that $G_{\delta,\text{o}}$ is odd,
\begin{align*}
    \abs{ \int_0^t G_{\delta,\text{o}}(s)\,\rd s }
     & =\abs{ \int_0^{2t_0-t} G_{\delta,\text{o}}(s)\,\rd s }
    \leq C\left(\left(\delta+t-t_0\right)^{\alpha-2}-\left(\delta+t_0\right)^{\alpha-2}\right).
\end{align*}
Thus for any $t\in[0,\Tf]$, we have
\begin{align}\label{eq:G_odd_bound}
    \abs{ \int_0^t G_{\delta,\text{o}}(s)\,\rd s }
    \leq C\left(\left(\delta+\abs{t-t_0}\right)^{\alpha-2}-\left(\delta+t_0\right)^{\alpha-2}\right).
\end{align}
Taking the estimates \eqref{eq:G_bdd_bound}, \eqref{eq:G_even_bound}, and \eqref{eq:G_odd_bound} together, from \eqref{eq:dF1dt_solution} we obtain
\begin{align}\label{eq:F1bound}
    \abs{\pa F(t)}\leq C \abs{ \int_0^t G_(s)\,\rd s }
    \leq C\left(1+\left(\delta+\abs{t-t_0}\right)^{\alpha-2}\right)\leq C\left(1+\left(\abs{P}^2+\delta^2\right)^{\frac{\alpha-2}{2}}\right).
\end{align}




\bibliographystyle{siamplain}
\bibliography{refs}

\begin{thebibliography}{10}

\bibitem{antoine_derivation_2021}
{\sc X.~Antoine, E.~Lorin, and Y.~Zhang}, {\em Derivation and analysis of computational methods for fractional {Laplacian} equations with absorbing layers}, Numerical Algorithms, 87 (2021), pp.~409--444.

\bibitem{antoine_ground_2016}
{\sc X.~Antoine, Q.~Tang, and Y.~Zhang}, {\em On the ground states and dynamics of space fractional nonlinear {Schrödinger}/{Gross}–{Pitaevskii} equations with rotation term and nonlocal nonlinear interactions}, Journal of Computational Physics, 325 (2016), pp.~74--97.

\bibitem{bao2002time}
{\sc W.~Bao, S.~Jin, and P.~A. Markowich}, {\em On time-splitting spectral approximations for the {Schr{\"o}dinger} equation in the semiclassical regime}, Journal of Computational Physics, 175 (2002), pp.~487--524.

\bibitem{carles2008semi}
{\sc R.~Carles}, {\em Semi-classical analysis for nonlinear {Schrodinger} equations}, World Scientific, 2008.

\bibitem{Chai2021fgaconv}
{\sc L.~Chai, J.~C. Hateley, E.~Lorin, and X.~Yang}, {\em On the convergence of frozen {G}aussian approximation for linear non-strictly hyperbolic systems}, Comm. Math. Sci., 19 (2021), pp.~585--606.

\bibitem{ChLo:19}
{\sc L.~Chai, E.~Lorin, and X.~Yang}, {\em Frozen {G}aussian approximation for the {D}irac equation in semi-classical regime}, SIAM J. Num. Anal., 57 (2019), pp.~2383--2412.

\bibitem{duo2016mass}
{\sc S.~Duo and Y.~Zhang}, {\em Mass-conservative {Fourier} spectral methods for solving the fractional nonlinear {Schr{\"o}dinger} equation}, Computers \& Mathematics with Applications, 71 (2016), pp.~2257--2271.

\bibitem{EnRu:03}
{\sc B.~Engquist and O.~Runborg}, {\em Computational high frequency wave propagation}, Acta Numer., 12 (2003), pp.~181--266.

\bibitem{frohlich2007boson}
{\sc J.~Fr{\"o}hlich, B.~L.~G. Jonsson, and E.~Lenzmann}, {\em Boson stars as solitary waves}, Communications in mathematical physics, 274 (2007), pp.~1--30.

\bibitem{gerard1997homogenization}
{\sc P.~G{\'e}rard, P.~A. Markowich, N.~J. Mauser, and F.~Poupaud}, {\em Homogenization limits and {Wigner} transforms}, Communications on Pure and Applied Mathematics: A Journal Issued by the Courant Institute of Mathematical Sciences, 50 (1997), pp.~323--379.

\bibitem{gosse2003two}
{\sc L.~Gosse, S.~Jin, and X.~Li}, {\em Two moment systems for computing multiphase semiclassical limits of the {Schr{\"o}dinger} equation}, Mathematical Models and Methods in Applied Sciences, 13 (2003), pp.~1689--1723.

\bibitem{hagedorn1998raising}
{\sc G.~A. Hagedorn}, {\em Raising and lowering operators for semiclassical wave packets}, Annals of Physics, 269 (1998), pp.~77--104.

\bibitem{hate2018fga}
{\sc J.~C. Hateley, L.~Chai, P.~Tong, and X.~Yang}, {\em Frozen {G}aussian approximation for 3-{D} elastic wave equation and seismic tomography}, Geophysical Journal International, 216 (2019), pp.~1394--1412.

\bibitem{He:81}
{\sc E.~J. Heller}, {\em Frozen {G}aussians: {A} very simple semiclassical approximation}, J. Chem. Phys., 75 (1981), pp.~2923--2931.

\bibitem{HeKl:84}
{\sc M.~F. Herman and E.~Kluk}, {\em A semiclassical justification for the use of non-spreading wavepackets in dynamics calculations}, Chem. Phys., 91 (1984), pp.~27--34.

\bibitem{hormander1987analysis}
{\sc L.~H{\"o}rmander and L.~H{\"o}rmander}, {\em The analysis of linear partial differential operators. III: Pseudo-Differential Operators}, Springer-Verlag, 1994.

\bibitem{HuYaOb2014}
{\sc Y.~Huang and A.~Oberman}, {\em Numerical methods for the fractional {Laplacian}: A finite difference-quadrature approach}, SIAM Journal on Numerical Analysis, 52 (2014), pp.~3056--3084.

\bibitem{Huang2022}
{\sc Z.~Huang, L.~Xu, and Z.~Zhou}, {\em Efficient frozen {G}aussian sampling algorithms for nonadiabatic quantum dynamics at metal surfaces}, ArXiv, abs/2206.02173 (2022).

\bibitem{jin2011mathematical}
{\sc S.~Jin, P.~Markowich, and C.~Sparber}, {\em Mathematical and computational methods for semiclassical {Schr{\"o}dinger} equations}, Acta Numerica, 20 (2011), pp.~121--209.

\bibitem{Ka:94}
{\sc K.~Kay}, {\em Integral expressions for the semi-classical time-dependent propagator}, J. Chem. Phys., 100 (1994), pp.~4377--4392.

\bibitem{Ka:06}
{\sc K.~Kay}, {\em The {H}erman-{K}luk approximation: {D}erivation and semiclassical corrections}, Chem. Phys., 322 (2006), pp.~3--12.

\bibitem{kirkpatrick2016fractional}
{\sc K.~Kirkpatrick and Y.~Zhang}, {\em Fractional {S}chr{\"o}dinger dynamics and decoherence}, Physica D: Nonlinear Phenomena, 332 (2016), pp.~41--54.

\bibitem{klein_numerical_2014}
{\sc C.~Klein, C.~Sparber, and P.~Markowich}, {\em Numerical study of fractional {Nonlinear} {Schr}\"odinger equations}, Proceedings of the Royal Society A: Mathematical, Physical and Engineering Sciences, 470 (2014), p.~20140364.
\newblock arXiv:1404.6262 [math].

\bibitem{laskin2000fractionalE}
{\sc N.~Laskin}, {\em Fractional quantum mechanics}, Physical Review E, 62 (2000), p.~3135.

\bibitem{laskin2000fractionalA}
{\sc N.~Laskin}, {\em {Fractional quantum mechanics and L{\'e}vy path integrals}}, Physics Letters A, 268 (2000), pp.~298--305.

\bibitem{laskin2002fractional}
{\sc N.~Laskin}, {\em {Fractional Schr{\"o}dinger equation}}, Physical Review E, 66 (2002), p.~056108.

\bibitem{LaSa:17}
{\sc C.~Lasser and D.~Sattlegger}, {\em Discretising the {Herman--Kluk} propagator}, Numerische Mathematik, 137 (2017).

\bibitem{lenzmann2007well}
{\sc E.~Lenzmann}, {\em Well-posedness for semi-relativistic {Hartree} equations of critical type}, Mathematical Physics, Analysis and Geometry, 10 (2007), pp.~43--64.

\bibitem{lions1993mesures}
{\sc P.-L. Lions and T.~Paul}, {\em Sur les mesures de {Wigner}}, Revista matem{\'a}tica iberoamericana, 9 (1993), pp.~553--618.

\bibitem{LuYa:11}
{\sc J.~Lu and X.~Yang}, {\em Frozen {G}aussian approximation for high frequency wave propagation}, Commun. Math. Sci., 9 (2011), pp.~663--683.

\bibitem{LuYa:CPAM}
{\sc J.~Lu and X.~Yang}, {\em Convergence of frozen {G}aussian approximation for high frequency wave propagation}, Comm. Pure Appl. Math., 65 (2012), pp.~759--789.

\bibitem{LuZh:2016FGAsh}
{\sc J.~Lu and Z.~Zhou}, {\em Improved sampling and validation of frozen {G}aussian approximation with surface hopping algorithm for nonadiabatic dynamics}, J Chem Phys., 145 (2016), p.~124109.

\bibitem{LuZh:2018FGAsh}
{\sc J.~Lu and Z.~Zhou}, {\em Frozen {G}aussian approximation with surface hopping for mixed quantum-classical dynamics: A mathematical justification of fewest switches surface hopping algorithms}, Math. Comput., 87 (2018), pp.~2189--2232.

\bibitem{lubich2008splitting}
{\sc C.~Lubich}, {\em On splitting methods for {S}chr{\"o}dinger-{Poisson} and cubic nonlinear {S}chr{\"o}dinger equations}, Mathematics of computation, 77 (2008), pp.~2141--2153.

\bibitem{martinez2002introduction}
{\sc A.~Martinez}, {\em An introduction to semiclassical and microlocal analysis}, vol.~994, Springer, 2002.

\bibitem{SwRo:09}
{\sc T.~Swart and V.~Rousse}, {\em A mathematical justification of the {H}erman-{K}luk propagator}, Commun. Math. Phys., 286 (2009), pp.~725--750.

\bibitem{tartar1990h}
{\sc L.~Tartar}, {\em H-measures, a new approach for studying homogenisation, oscillations and concentration effects in partial differential equations}, Proceedings of the Royal Society of Edinburgh Section A: Mathematics, 115 (1990), pp.~193--230.

\bibitem{wang2022lie}
{\sc W.~Wang, Y.~Huang, and J.~Tang}, {\em Lie-trotter operator splitting spectral method for linear semiclassical fractional {S}chr{\"o}dinger equation}, Computers \& Mathematics with Applications, 113 (2022), pp.~117--129.

\end{thebibliography}

\end{document}